\documentclass[11pt, reqno]{amsart}
\usepackage{amssymb,latexsym,amsmath,amsfonts, amsthm}

\usepackage{latexsym}

\usepackage{color}
\usepackage[mathscr]{eucal}
\usepackage{comment}

\usepackage[linkcolor=blue, citecolor=black]{hyperref}
\hypersetup{colorlinks=true, urlcolor=blue}
\usepackage[nameinlink]{cleveref}
\usepackage{enumerate}

\numberwithin{equation}{section}

\theoremstyle{definition}
\newtheorem{definition}{Definition}[section]

\theoremstyle{remark}
\newtheorem{remark}[definition]{Remark}
\theoremstyle{plain}
\newtheorem{theorem}[definition]{Theorem}
\newtheorem{lemma}[definition]{Lemma}
\newtheorem{proposition}[definition]{Proposition}
\newtheorem*{claim}{Claim}
\newtheorem{result}[definition]{Result}

\newtheorem{question}[definition]{Question}

\newtheorem{conjecture}[definition]{Conjecture}

\newcommand{\zbar}{\overline{z}}

\newcommand{\smoo}{\mathcal{C}}

\newcommand{\CC}{\mathbb{C}^2}
\newcommand{\cplx}{\mathbb{C}}

\newcommand{\C}{\mathbb{C}}

\begin{document}
	\title[ Surfaces with degenerated CR singularities]{Certain real surfaces in $\mathbb{C}^2$ with degenerated CR singularities}
	\author{Sushil Gorai, Suman Karak and Golam Mostafa Mondal}
\address{Department of Mathematics and Statistics, Indian Institute of Science Education and Research Kolkata,
		Mohanpur -- 741 246}
  	\email{sushil.gorai@iiserkol.ac.in, sushil.gorai@gmail.com}

\address{Department of Mathematics and Statistics, Indian Institute of Science Education and Research Kolkata,
		Mohanpur -- 741 246}
    \email{sk23rs056@iiserkol.ac.in, sumankarak7363@gmail.com}
    
	\address{Department of Mathematics, Indian Institute of Science, Bangalore-- 560012, India}
	
	\email{golammostafaa@gmail.com, golammondal@iisc.ac.in}

\begin{abstract}
	In this paper, we study the local polynomial convexity of certain smooth real surfaces   in \(\mathbb{C}^2\) with isolated CR singularity at the origin with higher-order of degeneracy. Under the assumption that the surface can be pulled back to a union of finitely many pairwise transverse totally real surfaces by a proper holomorphic map from $\mathbb{C}^2$ to $\mathbb{C}^2$, we obtain a normal form for such surfaces near the origin as  
    $\{(z,w)\in\mathbb{C}^2: w= \overline{z}^k+o(|z|^{k})\}$
    or
	$M_t
	:=
	\left\{
	(z,w)\in\mathbb{C}^2 :
	w=(z+t\overline{z})^k+o(|z|^k)
	\right\}$, 
   for some \(t>0\), where the parameter $t$ is a local biholomorphic invariant. We  focus on the surfaces with order of degeneracy $k\geq 3$. 
      We prove that \(M_t\) is locally polynomially convex at the origin if $t>cosec\left(\frac{\pi}{k}\right)$.
	On the other hand, for $0<t<\frac{1}{k-2}$, we will also show that \(M_t\) fails to be locally polynomially convex at the origin; and furthermore, a $(2k-3)$-parameter family of analytic discs attached to $M_t$ for $0<t<\min\left\{\sin\left(\frac{\pi}{k}\right),\frac{1}{k-2}\right\}$. 
	
\end{abstract}  	
\maketitle

	\section{Introduction }\label{S:intro}

Let $M$ be a real $\mathcal{C}^k$-smooth two-dimensional submanifold of $\mathbb{C}^{2}$. A point $p \in M$ is said to be a totally real point if the tangent space $T_{p}M$, viewed as a subspace of $\CC$, is not a complex subspace. A point $p\in M$ is called an isolated CR singularity if there exists $\delta>0$ such that each point $z\in M\cap B(p;\delta)\setminus\{p\}$ is a totally real point of $M$ but $T_pM$ is a complex subspace of $\cplx^2$. 
A \(\smoo^{k}\)-smooth real surface $M$ in \(\mathbb{C}^{2}\) with an isolated CR singularity at the origin is, under a holomorphic change of coordinates, of the form
\begin{equation}\label{surface}
M\cap B(0,\delta)
=
\{(z,f(z)):\,
f(z)=p_k(z,\bar z)+o(|z|^k)\},
\end{equation}
where \(p_k(z,\bar z)\) is a homogeneous polynomial of degree \(k\geq 2\) in \(z\) and \(\bar z\), for sufficiently small $\delta>0$. A surface of the form \eqref{surface}, with $p_k$ not identically zero, is said to have an \emph{isolated CR singularity of order \(k\)} at the origin.
In this paper we study such surfaces in the context of polynomial convexity. Recall that, for a compact set $K \subset \mathbb{C}^{n}$, the polynomially convex hull of $K$ is denoted by $\widehat{K}$ and defined by
\[\widehat{K}
=
\left\{ z_{0} \in \mathbb{C}^{n} : |p(z_{0})| \leq \max_{z \in K}|p(z)| \, \text{ for all holomorphic polynomial } p \text{ on } \mathbb{C}^{n} \right\}.\]
The set $K$ is said to be polynomially convex if $\widehat{K}=K$.
The theory of polynomial approximation, particularly the Oka--Weil theorem and the O'Farrell--Preskenis--Walsh theorem \cite{MR769507}, provides strong motivation for the study of polynomially convex sets. For a detailed exposition, we refer the reader to~\cite{gorai2025certainrealsurfacesmathbbc2}. 

The broad question is: {\em whether there exists $\delta>0$ such that $M\cap\overline{B(0;\delta)}$ is polynomially convex? If it is not polynomially convex, is there a family of analytic discs in the polynomially convex hull of $M\cap\overline{B(0;\delta)}$ for every $\delta>0$?} In this paper we will focus on a one-parameter family, up to local biholomorphism, of surfaces of the form \eqref{surface}. Before making it precise, we give a brief literature survey.

The case \(k=2\), the nondegenerate case, was first studied by Bishop~\cite{MR200476}. Under the non-degeneracy condition $\frac{\partial^2f}{\partial z\partial\zbar}(0)\neq 0$, he showed that, under a local biholomorphic change of coordinates near the origin, there is one parameter family of surfaces with the normal form
\[
\mathscr{M}_\gamma
=
\{(z,p(z,\bar{z})):\;\; p(z,\bar{z})= z\bar z+\gamma(z^2+\bar z^2)\},
\]
where \(\gamma\geq 0\) is a biholomorphic invariant, now known as the \emph{Bishop invariant}. Bishop studied the structure of the local polynomial hull for $0\leq\gamma<\dfrac{1}{2}$. Although Bishop's study was for any dimension, we will focus only on $\cplx^2$. 
According to the value of \(\gamma\), a classification for CR singularities into three types is possible:
\begin{itemize}
    \item \emph{elliptic} if \(0\leq \gamma<\frac{1}{2}\),
    \item \emph{parabolic} if \(\gamma=\frac{1}{2}\),
    \item \emph{hyperbolic} if \(\gamma>\frac{1}{2}\).
\end{itemize}
Forstneri\v{c} and Stout \cite{MR1115074} showed that a surface is locally polynomially convex at hyperbolic CR singularities. J\"{o}ricke \cite{MR1488338} studied the parabolic case.
For a detailed discussion of this, we refer the reader to \cite{MR1832326} (see also \cite{gorai2025certainrealsurfacesmathbbc2}). In the case of degenerate CR singularity at the origin, the surfaces cannot be parametrized by a single parameter. The local polynomial hull in this case was studied by \cite{Wiegerinck} (see also \cite{MR1128527}). Study of local polynomial convexity was initiated by Harris \cite{Harris1991} and continued by Bharali \cite{Bharali2005, Bharali2006, Bharali2011} for surfaces with degenerated CR singularities.
Gorai \cite{gorai2025certainrealsurfacesmathbbc2} studied surfaces with order of degeneracy three with the condition that the surface can be pulled back to a union of three totally real surfaces.
 In this paper, as in \cite{gorai2025certainrealsurfacesmathbbc2}, we put the following condition:
\begin{itemize}\label{geometric condition}
	\item [($\ast$)]
	\textit{$M$ is defined as in (\ref{surface})  and $	S=\{(z,w): w=p_k(z,\bar z)\},$ such that there exists a proper holomorphic map $\Phi:\mathbb{C}^2\to\mathbb{C}^2$ with $\Phi^{-1}(S)$ is the union of $k$ totally real planes in $\mathbb{C}^2$.}
\end{itemize}
We note that the condition $(\ast)$ is already inherent in Bishop's setting for $\gamma>0$; the surface $\mathscr{M}_0$ can not be pulled back to union of two totally real surfaces by a proper holomorphic map. Gorai~\cite{gorai2025certainrealsurfacesmathbbc2} reformulated Bishop's invariant by expressing the surface in the normal form
\[\{(z,w)\in\CC : w=(z+t\bar z)^2+o(|z|^2)\},\]
where $t=2\gamma>0$. This clearly shows that $t$ is a local biholomorphic invariant in this case.
For $k=3$, Thomas~\cite{MR1128527} showed that the geometric condition $(\ast)$ reduces to a condition on the coefficients of $p_3(z,\bar z)$. Under this condition, Gorai~\cite{gorai2025certainrealsurfacesmathbbc2} obtained a normal form for surfaces locally at the origin when the order of degeneracy of the CR singularity at the origin is three. After a biholomorphic change of coordinates, the surface $M$ is locally equivalent to
\begin{equation}\label{eq2}
M_t:=\left\{
(z,w):
w=z^2\bar{z}+tz\bar z^{\,2}+\frac{t^2}{3} \bar z^{\,3}+o(|z|^3)
\right\}
\end{equation}
for some $t>0$, where $t$ is a local biholomorphic invariant. The following result is established in \cite{gorai2025certainrealsurfacesmathbbc2}.

 \begin{result}[Gorai]\label{Gorai-polycvx}
Let $M$ be of the form \eqref{eq2}. Then
\begin{itemize}
    \item[(i)] $M_t$ is not locally polynomially convex at the origin for $0<t<1$. Moreover, for each $\delta>0$, the polynomial hull of $M\cap\overline{B(0,\delta)}$ contains an open neighborhood of the origin in $\mathbb{C}^2$ if $0<t<\frac{\sqrt{3}}{2}$. If $\frac{\sqrt{3}}{2}\le t<1$, then the polynomial hull of $M_t\cap\overline{B(0,\delta)}$ contains a one-parameter family of analytic discs passing through the origin with boundaries contained in $M$.

    \item[(ii)] If $t>\sqrt{\frac{3}{2}}$, then $M_t$ is locally polynomially convex at the origin.
\end{itemize}
\end{result}

 \noindent It is also observed in \cite{gorai2025certainrealsurfacesmathbbc2} that $M_t$ is locally biholomorphically equivalent to
\[
\left\{
(z,w)\in\CC:
w=(z+t\bar z)^3+o(|z|^3)
\right\}.
\]
 Gorai~\cite{gorai2025certainrealsurfacesmathbbc2} used his earlier results \cite{MR3123672} on union of three totally real planes in $\CC$ to  prove several results concerning the local polynomial convexity of the union of three totally real surfaces which are then used crucially in his proof of Result~\ref{Gorai-polycvx}. 
 In this paper we will generalize the above result for surfaces with CR singularity of higher order of degeneracy. 
Under condition $(\ast)$ for surfaces with CR singularity with higher-order degeneracy, there are several obstacles to determine whether surfaces of the form \eqref{surface} are locally polynomially convex at the origin. Firstly, for higher-order degeneracies, there is no result analogous to that of Thomas~\cite{MR1128527} that translates condition $(\ast)$ into an equivalent condition on the coefficients of $p_k(z,\bar z)$. Consequently, normal forms of the type obtained by Gorai in \cite{gorai2025certainrealsurfacesmathbbc2} are not readily available in this setting. Our first goal is to establish such a normal form.
 Let $M$ be a surface of the form (\ref{surface}) and $S=\{(z,w)\in \mathbb{C}^2 : w=p_k(z,\bar z)\},$ where
\[
p_k(z,\bar z)=\sum_{m=1}^{k}a_m\, z^{\,k-m}\bar z^{\,m}.
\]
Our first result gives a normal form for the surface $M$ of the form (\ref{surface}), under the geometric condition $(\ast).$ 
\begin{theorem}\label{coeff relation}
	Let $ \Phi:\mathbb{C}^2\to\mathbb{C}^2$ be a proper holomorphic map given by $\Phi(z,w)=(z,p_k(z,w)).$ Then $\Phi^{-1}(S)$ is the union of $k$ totally real planes if and only if $a_k\neq 0$ and
	\[
	\frac{a_{k-m}}{a_k}
	=
	\frac{\binom{k}{m}}{k^m}
	\left(\frac{a_{k-1}}{a_k}\right)^m,
	\qquad
	\text{for all } m=2,3,\dots,k-1.
	\]	
    Furthermore, these $k$ totally real planes are pairwise transverse if and only if $\left|\frac{a_{k-1}}{ka_k}\right|\neq 1$.
\end{theorem}

This is a generalization of a result by Thomas~\cite{MR1128527}. It allows us to consider a one-parameter family as in \cite{gorai2025certainrealsurfacesmathbbc2}.
We study the local polynomial convexity and the local polynomial hull of \eqref{surface} for higher orders \(k>3\) under condition ($\ast$), partially answering a question in \cite{gorai2025certainrealsurfacesmathbbc2}. 
Harris \cite{MR773068} showed that if \(M\) is \(C^\infty\)-smooth, then the homogeneous polynomial \(p_k(z,\bar z)\) admits a normal form. Hence, under the additional geometric condition, applying Theorem~\ref{coeff relation} yields a normal form which matches  that of \cite{gorai2025certainrealsurfacesmathbbc2}.
The preceding discussion leads to the following proposition, which establishes a normal form in our setting.

\begin{theorem}\label{L:NormalFrm_TotRlSurface}
Let \(M\) be a real surface in $\CC$ such that under a holomorphic change of coordinates, $M\cap B(0,\delta)=\{(z,w):\,w=p_k(z,\bar z)+o(|z|^k)\}$. If \(M\) satisfies the condition $(\ast)$, then after a biholomorphic change of coordinates near the origin, \(M\) can be written locally as
$
\left\{
(z,w)\in \mathbb{C}^2 :
w =\bar z^k + o(|z|^k)
\right\}
$
or 
$
\left\{
(z,w)\in \mathbb{C}^2 :
w = (z+t\bar z)^k + o(|z|^k) 
\right\},
$
where \(t>0\) is a biholomorphic invariant.
\end{theorem}
\begin{remark}
    If $a_k\neq 0$ and $a_j=0$ for all $j=1,2,\ldots,k-1$, then $p_k(z,\bar z)=a_k\bar z^k$. Hence, after a suitable scaling, we obtain the normal form
$\left\{(z,w)\in\mathbb{C}^2 : w=\bar z^k+o(|z|^k)\right\},$ which can be consider the surface corresponding to $t=\infty$.
Otherwise, we obtain the second normal form given in Theorem~\ref{L:NormalFrm_TotRlSurface}. 
\end{remark}
From now on, we will use the notation $M_t$ for surfaces with the normal of second type, i.e., 
\[
M_t:=
\left\{
(z,w)\in \mathbb{C}^2 :
w = (z+t\bar z)^k + o(|z|^k)
\right\}
\cap \overline{B(0,\delta)},
\]
for sufficiently small \(\delta>0\), and
\[
S_t=
\left\{
(z,w)\in \mathbb{C}^2 :
w =(z+t\bar z)^k
\right\}.
\]

\noindent Our next theorem is about the local polynomial hull of $M_t$ near the origin.

\begin{theorem}\label{ppp}
    For $k\geq 3$, and for every \(\delta>0\), \(M_t\) is not locally polynomially convex at the origin for $t\in \left(0,\frac{1}{k-2}\right)$. Moreover, the polynomial hull of $M_t$ contains a \((2k-3)\)-parameter family of analytic discs with boundary contained in \(M_t\) for $0<t<\min\left\{\sin\left(\frac{\pi}{k}\right),\frac{1}{k-2}\right\}$ and this family fills an open neighborhood of \(0\in \mathbb{C}^2\).
\end{theorem}

\noindent For  $k=3$ we recover the result due to gorai \cite[Theorem 1.11]{gorai2025certainrealsurfacesmathbbc2} (Part (i) of Result~\ref{Gorai-polycvx}).

 For $k=4$ we give another result on the local polynomial hull for $S_t$.
\begin{theorem}\label{thm-k4}
   Let $ \tilde{S_t}=\{(z,w)\in \mathbb{C}^2 : w=p_4(z,\bar z)=z^3\zbar+\frac{3t}{2}z^2\zbar^2+t^2z\zbar^3+\frac{t^3}{4}\zbar^4\}.$ Then, for $ t\in\left[\sin\left(\frac{\pi}{4}\right),1\right),$ the surface \(\tilde{S_t}\) is not locally polynomially convex at the origin and hence so is $S_t$.	
\end{theorem}

\medskip
We now focus our attention on all those parameters for which $M_t$ is locally polynomially convex at the origin and obtain the following result.
\begin{theorem}\label{T:With error term}
Fix \(k \in \mathbb{N}\) with \(k\geq3\). For each $t > cosec\left(\frac{\pi}{k}\right),$ \(M_t\) is locally polynomially convex at the origin.
\end{theorem}
\begin{remark}
For $k=3$, we obtain that $M_t$ is locally polynomially convex at the origin if $t>cosec\left(\frac{\pi}{3}\right)$ which is a sharper bound than that of Result~\ref{Gorai-polycvx}.
\end{remark}

 A few ideas about our proofs: 
\begin{itemize}
    \item For polynomial convexity we first use a proper holomorphic map $\Phi:\cplx^2\to\cplx^2$ such that $\Phi^{-1}(M_t)$ is the union of $k$ totally real surfaces intersecting at the origin. Therefore the problem reduces to the local polynomial convexity of  this union of $k$ totally real surfaces at the origin.
    
    \item There are very few results about the local polynomial convexity of the union of finitely many totally real planes; \cite{MR0960837} for union of two planes,  \cite{MR3123672,MR3206683,gorai2025certainrealsurfacesmathbbc2}. Such result two or three totally real surfaces are there \cite{SG2011,gorai2025certainrealsurfacesmathbbc2}. Hence, there is no such result that we can readily use in our case. By constructing a polynomial that maps each of these surfaces into a distinct sector with vertex at the origin in plane and by using Kallin's lemma, we address this difficulty. Special formulation of our planes plays a very crucial role here. 

    \item To prove theorems about polynomial hull we use Wiegerinck's results \cite{Wiegerinck}.
\end{itemize}

\begin{remark}
    A Bishop type dichotomy for nonparabolic case, i.e., except $t=1$ case seems to be possible. Just like Bishop's normal form here also we achieve a one parameter family, upto biholomorphism. Therefore, we argue that this is a right class of family with degenerated CR singularity having a trichotomy: $0<t<1$ corresponds to elliptic, $t=1$ corresponds to parabolic and $t>1$ corresponds to hyperbolic CR singularities 
    for every order of degeneracy.
\end{remark}
\noindent Finally, we pose a conjecture in the lines of \cite{gorai2025certainrealsurfacesmathbbc2}. 
\begin{conjecture}
The surface $M_t$ is locally polynomially convex if $t>1$. The local polynomial hull of $M_t$ contains an analytic disc for $t<1$. 
\end{conjecture}
\noindent We expect the behavior of local polynomially convex hull similar to parabolic case \cite{MR1488338} for $t=1$.

	\section{Preliminaries}\label{S:technical}
    In this section we collect some results from literature that will be used in our proofs. First we mention a couple of results concerning polynomial convexity; the first one is Kallin's lemma \cite{Kallin65} (see also \cite[Theorem 1.6.19]{Sto07}) and the second one is from Stout's book \cite[Theorem 1.3.9]{Sto07}.

\begin{result}\label{R:Sto_apprx}
Let \(X_1\) and \(X_2\) be compact polynomially convex subsets of \(\mathbb{C}^n\). Suppose that \(p\) is a polynomial such that
$\widehat{p(X_1)}\cap \widehat{p(X_2)} \subset \{0\}$.
If the set
$p^{-1}(0)\cap (X_1\cup X_2)$
is polynomially convex, then \(X_1\cup X_2\) is polynomially convex. 
\end{result}

\begin{result}\label{R:Sto_propr}
Let \(F:\mathbb{C}^n\to\mathbb{C}^n\) be a proper holomorphic map, and let \(X\subset\mathbb{C}^n\) be compact. Then \(X\) is polynomially convex if and only if \(F^{-1}(X)\) is polynomially convex.
\end{result}

Next, we give a very brief survey of Maslov-type index keeping focus on our requirement (for more detailed discussion see\cite{MR0894560}). Let \(\gamma:S^1\to M\) be a closed curve in a \(C^1\)-smooth \(n\)-dimensional real submanifold \(M\subset \mathbb C^n\). Assume that the pullback bundle \(\gamma^{*}TM\) is trivial. Then there exist continuous sections
\[
X_1,X_2,\ldots,X_n:S^1\to TM
\]
such that $\{X_1(\zeta),X_2(\zeta),\ldots,X_n(\zeta)\}$ forms a basis of \(T_{\gamma(\zeta)}M\) for every \(\zeta\in S^1\).
Define $d:S^1\to \mathbb C\setminus\{0\}$
by
\[
d(\zeta)=\det\bigl(X_1(\zeta),X_2(\zeta),\ldots,X_n(\zeta)\bigr).
\]
The winding number of \(d\) about the origin is independent of the choice of the sections \(X_1,\ldots,X_n\). This winding number is called the \emph{Maslov-type index} of \(\gamma\). In particular, let \(\gamma:S^1\to M\) be a closed curve in a \(C^1\)-smooth orientable totally real surface \(M\subset\mathbb C^2\). Since \(M\) is orientable, the bundle \(\gamma^{*}TM\) is trivial. Let \(X_1\) and \(X_2\) be continuous sections of \(\gamma^{*}TM\) forming a basis of \(T_{\gamma(\zeta)}M\) for every \(\zeta\in S^1\). The Maslov-type index of \(\gamma\) is then defined as the winding number of the map
\[
d:S^1\to\mathbb C\setminus\{0\},
\qquad
d(\zeta)=\det\bigl(X_1(\zeta),X_2(\zeta)\bigr).
\]
We denote it by $\operatorname{Ind}_{\mathcal M}(M,\gamma).$
An important property of \(\operatorname{Ind}_{\mathcal M}(M,\gamma)\) is that it depends only on the homology class of \(\gamma\) in \(M\). In particular, if \(\gamma_1\) and \(\gamma_2\) are homologous closed curves in \(M\), then
$
\operatorname{Ind}_{\mathcal M}(M,\gamma_1)
=
\operatorname{Ind}_{\mathcal M}(M,\gamma_2).
$
We put the definition of Maslov-type index of an isolated CR singularity as in \cite{Bharali2012} (see also \cite{gorai2025certainrealsurfacesmathbbc2}).

\begin{definition}
Let \(M\subset\mathbb C^2\) be a \(C^1\)-smooth orientable totally real surface with an isolated CR singularity at a point \(p\in M\). Let \(U_p\) be a contractible neighborhood of \(p\) such that $(U_p\cap M)\setminus\{p\}$ is totally real. The \emph{Maslov-type index} of the CR singularity \(p\), denoted by $\operatorname{Ind}_{\mathcal M}(M,p),$ is defined by
\[
\operatorname{Ind}_{\mathcal M}(M,p)
=
\operatorname{Ind}_{\mathcal M}(M,\gamma),
\]
where $\gamma:S^1\to (U_p\cap M)\setminus\{p\}$ is any simple closed curve winding once around \(p\).
\end{definition}
\noindent
As shown in \cite{gorai2025certainrealsurfacesmathbbc2}, the following result of Bharali~\cite{Bharali2012} plays a crucial role in computing the Maslov-type index in our setting.
\begin{result}\label{R:Bharali_MslovIndex} 
Let $p$ be a homogeneous polynomial of degree $k$ in $z$ and $\bar{z}$ such that $\{ z \in \mathbb{C} : \frac{\partial p}{\partial \bar{z}}(z, \bar{z}) = 0 \} = \{0\}.$ Define a polynomial $q$ in $z$ by $q(z) = \left. \frac{\partial p}{\partial \bar{z}}(z, \bar{z}) \right|_{\bar{z} = 1}.$ Then the Maslov-type index of the graph $\Gamma_p$ of $p$ at the origin is given by  
\[
\operatorname{Ind}_{\mathcal{M}}(\Gamma_p,0)= 2 \sum_{\zeta \in q^{-1}\{0\} \cap \mathbb{D}} \mu(\zeta) - (k - 1),
\]  
where $\mu(\zeta)$ denotes the multiplicity of $\zeta$ as a zero of the polynomial $q$.
\end{result}

Next, we recall several results of Wiegerinck~\cite{Wiegerinck} concerning the local polynomial hulls of certain graphs which will be useful in this paper.

\begin{result}[Wiegerinck, {\cite[Theorem 3.1]{Wiegerinck}}]\label{th1}
    Let $\varphi$ be a $C^k$-smooth, $k \geq 2$, function on a disc in $\mathbb{C}$ centered at the origin. Suppose that the graph of $\varphi$, denoted by $S_\varphi$, has an isolated CR singularity at the origin of Maslov-type index $j$, $0 < j < k$. If $\operatorname{Re}\left(\frac{\varphi(z, \bar{z})}{z^{j-1}}\right)$ is strictly subharmonic on a punctured neighborhood of the origin, then there exist analytic discs with boundary in $S_\varphi$.  
\end{result}

\begin{result}[Wiegerinck, {\cite[Theorem 3.3]{Wiegerinck}}]\label{llll}
    Let $F(z, \bar{z})$ be a homogeneous function of degree $k$ in $z$ and $\bar{z}$, and $\mathcal{C}^2$-smooth away from the origin in $\mathbb{C}$. Suppose that the origin is an isolated CR singularity of $S = \{w = F(z, \bar{z})\}$ and the Maslov-type index at $0 \in \mathbb{C}^2$ is $j$, $0 < j < k$. Assume that $\operatorname{Re}\left(\frac{F(z, \bar{z})}{z^{j-1}}\right)$ is a subharmonic but nowhere harmonic function on $\mathbb{C} \setminus \{0\}$. Then:
    \begin{enumerate}[(i)]
        \item $S$ is not locally polynomially convex at $0 \in \mathbb{C}^2$.
        \item For every $r > 0$, the polynomial hull of $S \cap \overline{B(0; r)}$ contains a $(2j - 1)$-parameter family of analytic discs with boundary in $S$ passing around zero if and only if the curve $\mathcal{C}: S^1 \to \mathbb{C}$ defined by
        \[
        \mathcal{C}(z) = \frac{F(z, \bar{z})}{z^k}
        \]
        has the following property:  
        \begin{itemize}
            \item[($\ast\ast$)] If, for two different points $z_1 \neq z_2$ on the unit circle, $\mathcal{C}(z_1) = \mathcal{C}(z_2)$, then $z_1$ and $z_2$ divide the unit circle into two segments of length at least $\frac{\pi}{k - j + 1}$. Moreover, if the Maslov-type index $j > 1$, this family will fill an open neighborhood of the origin in $\mathbb{C}^2$.
        \end{itemize}
        \item If Property $(\ast\ast)$ is not satisfied by the curve $\mathcal{C}$, then for every $r > 0$, the polynomial hull of $S \cap \overline{B(0; r)}$ contains at least a one-parameter family of analytic discs with boundary in $S$ passing through the origin.
    \end{enumerate}
\end{result}

\begin{result}[Wiegerinck, {\cite[Corollary 3.4]{Wiegerinck}}]\label{coro1}
		With the notation of Result~\ref{llll} but assume only that $Re\left(\frac{F(z,\bar{z})}{z^{j-1}}\right)$ is strictly subharmonic on an open sector containing points $z_1,z_2$  on the unit circle and a separating segment of length less than $\frac{\pi}{k-j+1}$. If the curve $\mathcal{C}$ satisfies $\mathcal{C}(z_1)=\mathcal{C}(z_2)$, then there exists a family of one-parameter analytic discs with boundary in $S$ passing through $0$. 
\end{result}

\begin{result}[Wiegerinck, {\cite[Theorem 3.5]{Wiegerinck}}]\label{p}
    Let $f(z, \bar{z}) = p_k(z, \bar{z}) + o(|z|^k)$ be a smooth function of class $\mathcal{C}^k$ on a disc $D(0; r)$, where $p_k$ is a homogeneous polynomial of degree $k$ in $z$ and $\bar{z}$. Suppose that the origin is an isolated CR singularity of $S = \{(z, w) \in \mathbb{C}^2 : w = p_k(z, \bar{z})\}$ and that the Maslov-type index at $0 \in \mathbb{C}^2$ is $j$, $0 < j < k$. Assume that $\operatorname{Re}\left(\frac{p_k(z, \bar{z})}{z^{j-1}}\right)$ is subharmonic but nowhere harmonic on $\mathbb{C} \setminus \{0\}$ and the curve $\mathcal{C}: S^1 \to \mathbb{C}$ defined by  
    \[
    \mathcal{C}(z) = \frac{p_k(z, \bar{z})}{z^k}
    \]
    has the Property $(**)$ from Result~\ref{llll}. Then, for every $\delta > 0$, the polynomial hull of $\Gamma_f \cap \overline{B(0; \delta)}$ contains a $(2j - 1)$-parameter family of analytic discs with boundary in $\Gamma_f$, whose union contains an open ball centered at the origin in $\mathbb{C}^2$ if $j > 1$.
\end{result}
The following result, due to Harris~\cite{MR773068}, provides a biholomorphic invariant normal form for a real surface in \(\mathbb{C}^2\) near an isolated CR singularity. We will  use this crucially.
\begin{result}\label{g}
Let $M$ be a $\smoo^k$-smooth real surface in $\mathbb{C}^2$ with an isolated CR-singularity at the origin of order $k.$ Then locally near the origin up to biholomorphism of $\mathbb{C}^2$, $M$ is of the form 
\begin{align*}
    \{(z,p_k(z,\bar{z})+o(|z|^k)\}
\end{align*}
  where $p_k$ is one of the following degree $k$ homogeneous polynomials
  \begin{enumerate}[(i)]
      \item $z^{k-\alpha_{0}}\bar{z}^{\alpha_{0}}$ for $k,\alpha_{0}\in \mathbb{Z},k\ge 2$ and $1\le \alpha_{0}\le k,$
      \item $z^{k-\alpha_{0}}\bar{z}^{\alpha_{0}}+\gamma z^{k-\alpha_{1}}\bar{z}^{\alpha_{1}}$ for $k,\alpha_{0},\alpha_1\in \mathbb{Z},k\ge 2,$ $1\le \alpha_{0}<\alpha_1\le k,$ and $\gamma>0$
      \item $z^{k-\alpha_{0}}\bar{z}^{\alpha_{0}}+\gamma z^{k-\alpha_{1}}\bar{z}^{\alpha_{1}}+\sum^{J}_{j=2}a_{k-\alpha_{j},\alpha_{j}}z^{k-\alpha_{j}}\bar{z}^{\alpha_{j}}$ for $k,J,\alpha_{0},\alpha_1,\cdots,\alpha_{J}\in \mathbb{Z},k\ge 2,$ $2\le J\le k-1$, $1\le \alpha_{0}<\alpha_1<\cdots<\alpha_{J}\le k,$ a real number $\gamma > 0$, and nonzero complex numbers $a_{k-\alpha_j,\alpha_j}$ $(2 \leq j \leq J)$ satisfying $0 < \text{Arg}(a_{k-\alpha_j,\alpha_j}) < \frac{2\pi}{k_j}$, where the integers $k_j$ are defined as follows: For each $j = 2, 3, \ldots, J$ let $\rho_j$ be the integer defined by
\begin{align*}
   0 < \rho_j < \alpha_1 - \alpha_0 \text{ and } (\alpha_j - \alpha_0) = m_j (\alpha_1 - \alpha_0) + \rho_j  
\end{align*}
for some integer $m_j$. Let $l_1 = a_1 - \alpha_0$. For given $l_1, \ldots, l_j$ let $l_{j+1} = \text{gcd}(l_j, \rho_{j+1})$ or $l_{j+1} = l_j$ if $\rho_{j+1} = 0$. For each $j = 2, \ldots, r$ let $k_j = \frac{l_{j-1}}{l_j}$.
\end{enumerate}
\end{result}

\noindent Just to make the reader comfortable we make the following remark, essentially from \cite{MR773068}.
\begin{remark}
For $k=4$, the possible forms of the polynomial $p_k$ are as follows:
\begin{enumerate}
    \item[(a)] $z^3\bar z$, $z^2\bar z^2$, $z\bar z^3$, $\bar z^4$.
    
    \item[(b)] $z^3\bar z+\gamma z^2\bar z^2$, $z^3\bar z+\gamma z\bar z^3$, $z^3\bar z+\gamma\bar z^4$, $z^2\bar z^2+\gamma\bar z^4$, $z^2\bar z^2+\gamma z\bar z^3$, $z\bar z^3+\gamma\bar z^4$.
    
    \item[(c)] $z^3\bar z+\gamma z^2\bar z^2+a_1z\bar z^3+a_2\bar z^4$, $z^3\bar z+\gamma z^2\bar z^2+a_3\bar z^4$, $z^3\bar z+\gamma z\bar z^3+a_4\bar z^4$, $z^2\bar z^2+\gamma z\bar z^3+a_5\bar z^4$.
\end{enumerate}
Similarly, the possible forms of $p_k$ can be determined for all $k\ge5$.
\end{remark}

\section{Normal forms}
In this section, we first prove Theorem~\ref{coeff relation}, which is the key step in obtaining the desired normal form.
\begin{proof}[Proof of Theorem~\ref{coeff relation}]
	Observe that
	\[\begin{aligned}
		\Phi^{-1}(S)
		&= \{(z,w)\in \mathbb{C}^2 : \Phi(z,w)\in S\} \\
		&= \{(z,w)\in \mathbb{C}^2 : p_k(z,w)=p_k(z,\bar z)\}.
	\end{aligned}
	\]
Factoring the left-hand side of $p_k(z,w)-p_k(z,\bar z)=0$, we obtain
$(w-\bar z)\, g(w,z,\bar z)=0,$ where
	\[
	\begin{aligned}
		g(w,z,\bar z)
		={}&
		a_k\sum_{m=0}^{k-1} w^m \bar z^{\,k-m-1}
		+a_{k-1}\left(\sum_{m=0}^{k-2} w^m \bar z^{\,k-m-2}\right)z \\
		&\quad + \cdots
		+a_3(w^2+w\bar z+\bar z^2)z^{k-3}
		+a_2(w+\bar z)z^{k-2}
		+a_1 z^{k-1}.
	\end{aligned}
	\]
	
	The component $\{w=\bar z\}$ is always a totally real plane. We now determine the remaining factors of \(g\).
	Assume that there exist complex numbers $\alpha_1,\dots,\alpha_{k-1},\beta_1,\dots,\beta_{k-1},$
	such that
	\begin{equation}\label{E:Others_roots_g_polished}
		g(w,z,\bar z)
		=
		a_k\prod_{j=1}^{k-1}
		\bigl(w-\alpha_j z-\beta_j\bar z\bigr).
	\end{equation}
	
	Comparing coefficients in \eqref{E:Others_roots_g_polished}, we obtain a collection of relations among the parameters \(\alpha_j\), \(\beta_j\), and the coefficients of \(p_k\). For distinct indices \(j_1,j_2,\ldots,j_{k-1}\in\{1,2,\ldots,k-1\}\), comparison of coefficients in \eqref{E:Others_roots_g_polished} yields the following equations.
	
	First, comparing the coefficients corresponding to the terms involving one \(\alpha\)-factor yields
	\begin{equation}\label{eq:p_polished}
		\left\{ 	\begin{aligned}
			\sum_{\substack{1\le j_1\le k-1}}\alpha_{j_1}&=-\frac{a_{k-1}}{a_k},\\
			\sum_{\substack{1\le j_1,j_2\le k-1}}\alpha_{j_1}\beta_{j_2}&=\frac{a_{k-1}}{a_k},\\
            \sum_{\substack{1\le j_1\le k-1\\1\le j_2<j_3\le k-1}}\alpha_{j_1}\beta_{j_2}\beta_{j_3}&=-\frac{a_{k-1}}{a_k} ,\\
			\vdots\\
			\sum_{\substack{1\le j_1\le k-1\\1\le j_2<...<j_{k-1}\le k-1}}\alpha_{j_1}\beta_{j_2}...\beta_{j_{k-1}}&=(-1)^{k-1}\frac{a_{k-1}}{a_k}.
		\end{aligned}\right.
	\end{equation}
	Similarly, comparing coefficients involving two \(\alpha\)-factors gives
	\begin{equation}\label{eq:q_polished}
		\left\{	\begin{aligned}
			\sum_{\substack{1\le j_1<j_2\le k-1}}\alpha_{j_1}\alpha_{j_2}&=\frac{a_{k-2}}{a_k},\\
			\sum_{\substack{1\le j_1<j_2\le k-1\\1\le j_3\le k-1}}\alpha_{j_1}\alpha_{j_2}\beta_{j_3}&=-\frac{a_{k-2}}{a_k}, \\
			\sum_{\substack{1\le j_1<j_2\le k-1\\1\le j_3<j_4\le k-1}}\alpha_{j_1}\alpha_{j_2}\beta_{j_3}\beta_{j_4}&=\frac{a_{k-2}}{a_k}, \\
			\vdots\\
			\sum_{\substack{1\le j_1<j_2\le k-1\\1\le j_3<...<j_{k-1}\le k-1}}\alpha_{j_1}\alpha_{j_2}\beta_{j_3}...\beta_{j_{k-1}}&=(-1)^{k-1}\frac{a_{k-2}}{a_k}.
		\end{aligned}  \right.
	\end{equation}
	
	Proceeding inductively, we obtain
	\begin{equation}\label{eq:s_polished}
		\left\{	\begin{aligned}
			\sum_{\substack{1\le j_1<...<j_{k-2}\le k-1}}\alpha_{j_1}\alpha_{j_2}...\alpha_{j_{k-2}}&=(-1)^{k-2}\frac{a_{2}}{a_k}, \\
			\sum_{\substack{1\le j_1<...<j_{k-2}\le k-1\\1\le j_{k-1}\le k-1}}\alpha_{j_1}\alpha_{j_2}...\alpha_{j_{k-2}}\beta_{j_{k-1}}&=(-1)^{k-1}\frac{a_{2}}{a_k}.
		\end{aligned}  \right.
	\end{equation}
	
	Finally,
	\begin{equation}\label{eq:t_polished}
		\alpha_1\alpha_2\cdots\alpha_{k-1}
		=
		(-1)^{k-1}\frac{a_1}{a_k}.
	\end{equation}
	
	Comparing the coefficients involving only the \(\beta_j\)'s yields
	\begin{equation}\label{eq:a_polished}
		\left\{ 	\begin{aligned}
			\sum_{\substack{1\le j_1\le k-1}}\beta_{j_1}&=-1,\\
			\sum_{\substack{1\le j_1<j_2\le k-1}}\beta_{j_1}\beta_{j_2}&=1,\\
			\vdots\\
			\sum_{\substack{1\le j_1<...<j_{k-1}\le k-1}}\beta_{j_1}\beta_{j_2}...\beta_{j_{k-1}}&=(-1)^{k-1}.
		\end{aligned} \right.
	\end{equation}
	
	It follows from \eqref{eq:a_polished} that $\beta_1,\dots,\beta_{k-1}$ are precisely the roots of $x^{k-1}+x^{k-2}+\cdots+x+1=0.$
	Let $r=e^{2\pi i/k};$ then $\beta_j=r^j$ for $j=1,\dots,k-1.$
	Using the relations between the roots and coefficients of the polynomial
	\[
	x^{k-1}+x^{k-2}+\cdots+x+1,
	\]
	the identities in \eqref{eq:p_polished} reduce to
	\begin{equation}\label{eq:u_polished}
		\left\{
		\begin{aligned}
			\sum_{j=1}^{k-1}\alpha_j
			&=-\frac{a_{k-1}}{a_k},\\
			\sum_{j=1}^{k-1}\alpha_j(1+\beta_j)
			&=-\frac{a_{k-1}}{a_k},\\
			\sum_{j=1}^{k-1}\alpha_j(1+\beta_j+\beta_j^2)
			&=-\frac{a_{k-1}}{a_k},\\
			&\vdots\\
			\sum_{j=1}^{k-1}\alpha_j
			\left(1+\beta_j+\cdots+\beta_j^{k-2}\right)
			&=-\frac{a_{k-1}}{a_k}.
		\end{aligned}
		\right.
	\end{equation}
	
	Subtracting consecutive equations, we obtain
	\[
	\sum_{j=1}^{k-1}\alpha_j\beta_j^m=0,
	\qquad
	m=1,\dots,k-2.
	\]
	Hence,
	\[
	\begin{pmatrix}
		1 & 1 & \cdots & 1 \\
		\beta_1 & \beta_2 & \cdots & \beta_{k-1} \\
		\beta_1^2 & \beta_2^2 & \cdots & \beta_{k-1}^2 \\
		\vdots & \vdots & \ddots & \vdots \\
		\beta_1^{k-2} & \beta_2^{k-2} & \cdots & \beta_{k-1}^{k-2}
	\end{pmatrix}
	\begin{pmatrix}
		\alpha_1\\
		\alpha_2\\
		\vdots\\
		\alpha_{k-1}
	\end{pmatrix}
	=
	\begin{pmatrix}
		-\dfrac{a_{k-1}}{a_k}\\
		0\\
		\vdots\\
		0
	\end{pmatrix}.
	\]
	
	Denoting this system by \(AX=b\), we observe that \(A\) is a Vandermonde matrix. Since the numbers
	$\beta_1,\dots,\beta_{k-1}$
	are distinct, the matrix \(A\) is invertible.
	
	Let $A^{-1}=(b_{ij})_{i,j=1}^{k-1}.$ Then entries of the inverse Vandermonde matrix are given by
	\[
	b_{ij}
	=
	(-1)^{k-j-1}
	\frac{
		e_{k-j-1}
		(\beta_1,\dots,\widehat{\beta_i},\dots,\beta_{k-1})
	}{
		\prod_{\substack{m=1\\m\neq i}}^{k-1}
		(\beta_i-\beta_m)
	},
	\]
	where
	$e_m(x_1,\dots,x_\ell)
	=
	\sum_{1\leq j_1<\cdots<j_m\leq \ell}
	x_{j_1}\cdots x_{j_m}$
	denotes the \(m\)-th elementary symmetric polynomial.
	To determine \(\alpha_j\), it suffices to compute $	b_{11},b_{21},\dots,b_{(k-1)1}.$ For \(j=1,\dots,k-1\), we compute
	\begin{align}
		b_{j1}
		&=
		(-1)^{k-2}
		\frac{
			e_{k-2}(\beta_1,\dots,\widehat{\beta_j},\dots,\beta_{k-1})
		}{
			\prod_{\substack{m=1\\m\neq j}}^{k-1}
			(\beta_j-\beta_m)
		}
		\nonumber\\
		&=
		(-1)^{k-2}
		\prod_{\substack{m=1\\m\neq j}}^{k-1}
		\frac{\beta_m}{\beta_j-\beta_m}
		\nonumber\\
		&=
		(-1)^{k-2}
		\prod_{\substack{m=1\\m\neq j}}^{k-1}
		\frac{r^m}{r^j-r^m}
		\nonumber\\
		&=
		(-1)^{k-2}
		\frac{r^j-1}{
			(r-1)(r^2-1)\cdots(r^{k-1}-1)
		}.
		\label{eq:bj1_polished}
	\end{align}
	
	Since $1,r,r^2,\dots,r^{k-1}$ are the roots of \(x^k-1=0\), it follows that
	$0,r-1,r^2-1,\dots,r^{k-1}-1$ are the roots of $(x+1)^k-1.$
	Therefore, $(r-1)(r^2-1)\cdots(r^{k-1}-1)=(-1)^{k-1}k.$ Substituting this into \eqref{eq:bj1_polished}, we obtain $b_{j1}=\frac{1-r^j}{k}.$
	Consequently,
	\[
	\alpha_j
	=
	-\frac{a_{k-1}}{a_k}b_{j1}
	=
	\frac{a_{k-1}}{a_k}\frac{r^j-1}{k},
	\qquad
	j=1,\dots,k-1.
	\]
	
	Substituting these expressions into
	\eqref{eq:p_polished}--\eqref{eq:t_polished}, we obtain
	\[
	(-1)^m\frac{a_{k-m}}{a_k}
	=
	\sum_{1\leq j_1<\cdots<j_m\leq k-1}
	\alpha_{j_1}\cdots\alpha_{j_m}.
	\]
	Using the elementary symmetric functions of the \(k\)-th roots of unity, this simplifies to
	\[
	(-1)^m\frac{a_{k-m}}{a_k}
	=
	\left(\frac{a_{k-1}}{a_k}\right)^m
	(-1)^m
	\frac{\binom{k}{m}}{k^m}.
	\]
	
	Hence,
	\[
	\frac{a_{k-m}}{a_k}
	=
	\left(\frac{a_{k-1}}{a_k}\right)^m
	\frac{\binom{k}{m}}{k^m},
	\qquad
	m=0,1,\dots,k-1.
	\]

Conversely, suppose that $a_k\neq 0$ and
$\frac{a_{k-m}}{a_k}
=
\left(\frac{a_{k-1}}{a_k}\right)^m
\frac{\binom{k}{m}}{k^m}$
for $m=0,1,\ldots,k-1$.
Then
\begin{align*}
p_k(z,w)
&=
a_1z^{k-1}w+a_2z^{k-2}w^2+\cdots+a_{k-1}zw^{k-1}+a_kw^k\\
&=
a_k\sum_{m=0}^{k-1}
\frac{a_{k-m}}{a_k}z^mw^{k-m}\\
&=
a_k\sum_{m=0}^{k-1}
\binom{k}{m}
\left(\frac{a_{k-1}}{ka_k}z\right)^m
w^{k-m}.
\end{align*}

Now, $p_k(z,w)=p_k(z,\bar z)$ gives,
\[
\sum_{m=0}^{k}
\binom{k}{m}
\left(\frac{a_{k-1}}{ka_k}z\right)^m
w^{k-m}
=
\sum_{m=0}^{k}
\binom{k}{m}
\left(\frac{a_{k-1}}{ka_k}z\right)^m
\bar z^{\,k-m},
\]
which implies
\[
\left(
w+\frac{a_{k-1}}{ka_k}z
\right)^k
=
\left(
\bar z+\frac{a_{k-1}}{ka_k}z
\right)^k.
\]

Therefore
\[
w+\frac{a_{k-1}}{ka_k}z
=
r^j
\left(
\bar z+\frac{a_{k-1}}{ka_k}z
\right),
\qquad
j=0,1,\ldots,k-1,
\]
where $r=e^{2\pi i/k}$. Consequently, $w=(r^j-1)\frac{a_{k-1}}{ka_k}z+r^j\bar z$ for $j=0,1,\ldots,k-1.$

Hence $\Phi^{-1}(S)
=
\{(z,w):p_k(z,w)=p_k(z,\bar z)\}
=
\bigcup_{j=0}^{k-1}X_j,$
where 
\[X_j
=
\left\{
(z,w)\in\mathbb C^2:
w=
(r^j-1)\frac{a_{k-1}}{ka_k}z
+r^j\bar z
\right\}\]
is a totally real plane for each $j=0,1,\ldots,k-1.$ This proves the required necessary and sufficient conditions.

Also, if $\left|\frac{a_{k-1}}{ka_k}\right|\neq 1$, then the planes $X_j$ are pairwise transverse. Conversely, suppose that
$\left\{(z,w)\in\mathbb{C}^2:w=\alpha_jz+\beta_j\bar z\right\}$ and
$\left\{(z,w)\in\mathbb{C}^2:w=\alpha_lz+\beta_l\bar z\right\}$
intersect transversally at the origin. Since both are $2$-dimensional real subspaces of $\C^2$, transversality at the origin is equivalent to their intersection being precisely the origin. Thus, the equation
$\alpha_jz+\beta_j\bar z=\alpha_lz+\beta_l\bar z$
has the unique solution $z=0$. Equivalently,
$\left|\frac{\alpha_j-\alpha_l}{\beta_j-\beta_l}\right|\neq 1$.
Substituting the expressions for $\alpha_j$, $\alpha_l$, $\beta_j$, and $\beta_l$ yields
$\left|\frac{a_{k-1}}{ka_k}\right|\neq 1$.
\end{proof}
The following proposition shows that the proper holomorphic map in condition $(\ast)$ is unique up to invertible $\C$-linear transformations.
\begin{proposition}
     Let $S=\{(z,w)\in \mathbb{C}^2 : w=p_k(z,\bar z)\},$ where $p_k(z,w)=\sum_{m=1}^{k}a_m\, z^{\,k-m}w^{\,m}$ with $a_k\neq 0$ and $\frac{a_{k-m}}{a_k}=\frac{\binom{k}{m}}{k^m}\left(\frac{a_{k-1}}{a_k}\right)^m,\text{for all } m=2,3,\dots,k-1.$ Let \(\Psi:\mathbb C^2\to\mathbb C^2\) be a holomorphic map of the form $\Psi(z,w)=(z,Q(z,w)),$ where \(Q\) is a homogeneous polynomial. Assume that
	\(\Psi^{-1}(S)\) contains a totally real plane. Then \(\Psi=\Phi\) as defined in Theorem~\ref{coeff relation}, up to an invertible complex linear
	transformation.
\end{proposition}

\begin{proof}
	Let $\Gamma_1=\{(z,w)\in\C^2:w=\alpha_1z+\beta_1\bar z\}\subset\Psi^{-1}(S)$
	where $\beta_1\ne 0.$ Consider the invertible complex linear map $L:\C^2\longrightarrow\C^2$ by $L(z,w)=\left(z,\frac{w-\alpha_1z}{\beta_1}\right).$
	Then
	\[
	L(\Gamma_1)=\{(z,w)\in\C^2:w=\bar z\}.
	\]
	
	  Let $\widetilde{\Psi}:=\Psi\circ L^{-1}.$ Since $\Gamma_1=L^{-1}\bigl(\{(z,w)\in\C^2:w=\bar z\}\bigr),$ we obtain
	\[
	\widetilde{\Psi}\bigl(\{(z,w)\in\C^2:w=\bar z\}\bigr)
	=\Psi(\Gamma_1)\subset S.
	\]
	Hence,
	\[
	Q(z,\alpha_1z+\beta_1\bar z)=p_k(z,\bar z),
	\qquad z\in\C.
	\]
	
	Now define $F(z,w):=Q(z,\alpha_1z+\beta_1w)-p_k(z,w).$ Then \(F\) is a holomorphic polynomial on \(\C^2\) satisfying $F(z,\bar z)=0$ for all $z\in\C.$
	Since $\{(z,w)\in\C^2:w=\bar z\}$ is a maximally totally real subspace of \(\C^2\), the identity principle implies that \(F\equiv0\). Consequently,
	\[
	Q(z,\alpha_1z+\beta_1w)=p_k(z,w),
	\qquad (z,w)\in\C^2.
	\]
	Replacing \(w\) by \((w-\alpha_1z)/\beta_1\), we obtain $Q(z,w)=p_k\left(z,\frac{w-\alpha_1z}{\beta_1}\right).$ Therefore,
	\[
	\Psi=\Phi\circ L,
	\]
	where \(\Phi(z,w)=(z,p_k(z,w))\). Hence \(\Psi\) coincides with \(\Phi\) up to an invertible complex linear transformation.
\end{proof}

We are now in a position to provide a proof of Theorem~\ref{L:NormalFrm_TotRlSurface}.

\begin{proof}[Proof of Theorem~\ref{L:NormalFrm_TotRlSurface}]
We shall use the normal form obtained in Result~\ref{g}. If $a_k\neq 0, a_j=0$ $\forall j=1,2,\cdots,k-1,$ then $p_k(z,\zbar)=a_k\zbar^k$. After a biholomorphic change of coordinates near the origin, \(M\) can be written locally as
\[
M \cap B(0,\delta)
=
\left\{
(z,w)\in \mathbb{C}^2 :
 w =\bar z^k + o(|z|^k)
\right\}.
\]
If not all of the $a_j$ vanish, then the condition $(\ast)$ or equivalently $a_k\neq 0,$
\[
\frac{a_{k-m}}{a_k}
=
\left(\frac{a_{k-1}}{a_k}\right)^m
\frac{\binom{k}{m}}{k^m},
\qquad
m=1,2,\ldots,k-1
\]
holds only if $p_k(z,\zbar)$ is of the form $(iii)$ in Result~\ref{g}.
The above condition can be written in the form,
\[
\frac{a_{k-m+1}}{a_{k-m}}
=
\left(\frac{a_k}{a_{k-1}}\right)
\frac{k\binom{k}{m-1}}{\binom{k}{m}},
\qquad
m=1,2,\ldots,k-1.
\]
Now, put $a_1=1, a_2=\gamma,$ and a straightforward computation  yields
\begin{align*}
a_m
&=
\gamma^{m-1}
\frac{\binom{k}{k-m}}
{\binom{k}{k-2}^{\,m-1}}
\binom{k}{k-1}^{\,m-2} \\
&=
\gamma^{m-1}
\left(
\frac{\binom{k}{k-1}^{\,m-2}}
{\binom{k}{k-2}^{\,m-1}}
\right)
\binom{k}{k-m} \\
&=
\left(\frac{2\gamma}{k-1}\right)^{m-1}
\frac{\binom{k}{m}}{k},
\qquad
m=1,2,\ldots,k.
\end{align*}

Therefore,
\[
p_k(z,\bar z)
=
\sum_{m=1}^{k}
\left(\frac{2\gamma}{k-1}\right)^{m-1}
\frac{\binom{k}{m}}{k}
\, z^{k-m}\bar z^{\,m},
\qquad
\gamma\in(0,\infty).
\]

Now consider the manifold
$M_\gamma
=
\left\{
(z,p_k(z,\bar z)+o(|z|^k))
:\ |z|<\delta
\right\},$
where \(\gamma\) is a biholomorphic invariant.
Define a biholomorphic change of coordinates $\sigma:\mathbb{C}^2\to\mathbb{C}^2$ by
$\sigma=\sigma_1\circ\sigma_2,$
where
\[
\sigma_1(z,w)=(z,z^k+w),
\qquad
\sigma_2(z,w)=\left(z,\frac{2k\gamma}{k-1}w\right).
\]

Then, locally near the origin,
\[
\sigma(M_\gamma)
=
\left\{
\left(
z,
\left(z+\frac{2\gamma}{k-1}\bar z\right)^k
+o(|z|^k)
\right)
\right\}.
\]
Equivalently,
\[
\sigma(M_\gamma)
=
\left\{
(z,(z+t\bar z)^k+o(|z|^k))
\right\},
\]
where $t=\frac{2\gamma}{k-1}$ is a biholomorphic invariant.
\end{proof}

\section{Local polynomial hull}\label{local polynomial hull}

Before proving our main theorems, let us clarify a point that may be confusing to the reader. We know that the polynomial convexity of a compact set \(K\subset\mathbb{C}^n\) is invariant under any biholomorphism \(F:\mathbb{C}^n\to\mathbb{C}^n\). We have
$
S_t=\left\{(z,w)\in \mathbb{C}^2 : w =(z+t\bar z)^k\right\}.
$
Also, from the proof of Theorem~\ref{L:NormalFrm_TotRlSurface}, after a biholomorphic change of coordinates, \(S_t\) becomes
\[\left\{(z,w):w= p_k(z, \bar{z}) = \sum_{m=1}^k t^{m-1} \frac{\binom{k}{m}}{k} z^{k-m} \bar{z}^m\right\}
=:\widetilde{S_t}.\]

Therefore, polynomial convexity remains invariant whether one use
$
p_k(z,\bar z)=(z+t\bar z)^k
$
or
$
p_k(z,\bar z)=\sum_{m=1}^k t^{m-1}\frac{\binom{k}{m}}{k}z^{k-m}\bar z^m.
$
In other words, \(S_t\) is locally polynomially convex at the origin if and only if \(\widetilde{S_t}\) is locally polynomially convex at the origin for a fixed $t>0$.

Hence, from now on, we shall use either of these two expressions for \(p_k(z,\bar z)\), depending on which one simplifies the computations.

\begin{theorem}\label{T:Not_Poly_Convx}
For $k\ge3,$ the surface \(S_t\) is not locally polynomially convex at the origin if $t\in \left(0,\frac{1}{k-2}\right).$
\end{theorem}
\begin{proof}
First, we compute the Maslov-type index for $\tilde{S_t}$ at the origin in $\mathbb{C}^2$. Since 
\[
p_k(z, \bar{z}) = \sum_{m=1}^k t^{m-1} \frac{\binom{k}{m}}{k} z^{k-m} \bar{z}^m,
\]
we have
\begin{align*}
    \frac{\partial p_k}{\partial \bar{z}} &= \sum_{m=1}^k m t^{m-1} \frac{\binom{k}{m}}{k} z^{k-m} \bar{z}^{m-1} \\
    &= (z + t \bar{z})^{k-1}.
\end{align*}
Note that for $t \neq 1$, $\left(\frac{\partial p_k}{\partial \bar{z}}\right)^{-1}(0) = \{0\}$.

Now, define
\[
    q_k(z) = \frac{\partial p_k}{\partial \bar{z}}(z, 1) = (z + t)^{k-1}.
\]
By Result~\ref{R:Bharali_MslovIndex}, the Maslov-type index of $S_t$ at the origin is:
\begin{align*}
    \text{Ind}_{\mathcal{M}}(\tilde{S_t},0) = 2\left(\sum \{ \mu(\xi) : \xi \in q_k^{-1}\{0\} \cap \mathbb{D} \}\right) - (k - 1),
\end{align*}
where $\mu(\xi)$ denotes the multiplicity of $\xi$ as a zero of the polynomial $q_k$. Therefore,
\begin{align*}
    \text{Ind}_{\mathcal{M}}(\tilde{S_t},0) = (k - 1) \text{ if } t < 1 \text{ and } -(k - 1) \text{ if } t > 1.
\end{align*}

\noindent Now we check subharmonicity of $\phi_k = \text{Re}\left(\frac{p_k(z, \bar{z})}{z^{k-2}}\right)$.
We compute:
\begin{align*}
    \frac{p_k(z, \bar{z})}{z^{k-2}} &= z \bar{z} + \sum_{m=2}^k t^{m-1} \frac{\binom{k}{m}}{k} \frac{\bar{z}^{2m-2}}{|z|^{2m-4}}.
\end{align*}
Therefore,
\begin{align*}
    \phi_k(z, \bar{z}) &= z \bar{z} + \sum_{m=2}^k t^{m-1} \frac{\binom{k}{m}}{2k} \left(\frac{z^m}{\bar{z}^{m-2}} + \frac{\bar{z}^m}{z^{m-2}}\right).
\end{align*}
Hence,
\begin{align*}
    \frac{\partial \phi_k}{\partial \bar{z}} &= z + \sum_{m=2}^k t^{m-1} \frac{\binom{k}{m}}{2k} \left( -(m - 2) \frac{z^m}{\bar{z}^{m-1}} + m \frac{\bar{z}^{m-1}}{z^{m-2}} \right),
\end{align*}
and
\begin{align*}
    \frac{\partial^2 \phi_k}{\partial \bar{z} \partial z} &= 1 - \sum_{m=2}^k t^{m-1} m (m - 2) \frac{\binom{k}{m}}{2k} \left( \frac{z^{m-1}}{\bar{z}^{m-1}} + \frac{\bar{z}^{m-1}}{z^{m-1}} \right).\\
    &\geq 1-\sum_{m=2}^k m (m - 2) \frac{\binom{k}{m}}{k}t^{m-1}\\
    &= (1-(k-2)t)(1+t)^{k-2}.
\end{align*}
Hence,
\begin{align*}
 \frac{\partial^2 \phi_k}{\partial \bar{z} \partial z} > 0 \text{ on } \mathbb{C} \setminus \{0\} \text{ if } t \in \left(0, \frac{1}{k - 2}\right).    
\end{align*}
Therefore, by Result~\ref{th1}, \(\widetilde{S_t}\) is not locally polynomially convex at the origin for
\(t\in\left(0,\frac{1}{k-2}\right)\).
Consequently, \(S_t\) is also not locally polynomially convex at the origin.
\end{proof}

\begin{lemma}\label{L:Roots_g_t}
Let $k\ge3$ and $g_t(z)
=
\sum_{m=1}^{k}
t^{m-1}\frac{\binom{k}{m}}{k}z^{2m}, t\geq 0.$
Then, for each \(a\in \partial \mathbb{D}\), the equation $g_t(z)=g_t(a)$ has exactly two solutions on \(\partial\mathbb{D}\) if and only if $0\leq t<\sin\left(\frac{\pi}{k}\right).$
\end{lemma}
\begin{proof}
If \(t=0\), then the equation $g_t(z)-g_t(a)=0$ has exactly two solutions, namely \(a\) and \(-a\). Now assume \(t\neq 0\). Observe that \(a\) and \(-a\) are already roots of $g_t(z)-g_t(a).$ Thus, it suffices to determine the range of \(t\) for which there are no additional roots on \(\partial\mathbb{D}\). We first rewrite \(g_t\) as follows:
\[
\begin{aligned}
g_t(z)
&=
\sum_{m=1}^{k}
t^{m-1}\frac{\binom{k}{m}}{k}z^{2m} \\
&=
\frac{1}{kt}
\sum_{m=1}^{k}
\binom{k}{m}(tz^2)^m \\
&=
\frac{1}{kt}
\left[(tz^2+1)^k-1\right].
\end{aligned}
\]

Let \(a=e^{i\theta}\), where \(\theta\in \mathbb{R}\). Then
\[
\begin{aligned}
g_t(z)=g_t(a)
&\iff
(tz^2+1)^k=(ta^2+1)^k \\
&\iff
\left(z^2+\frac{1}{t}\right)^k
=
\left(a^2+\frac{1}{t}\right)^k.
\end{aligned}
\]

Therefore,
$
z^2+\frac{1}{t}
=
r^m
\left(a^2+\frac{1}{t}\right),
 m=0,1,\dots,k-1,
$
where $r=e^{\frac{2\pi i}{k}}.$
Equivalently,
\[
z^2
=
r^m
\left(a^2+\frac{1}{t}\right)
-\frac{1}{t},
\qquad m=0,1,\dots,k-1.
\]

Our goal is to determine the values of \(t\) for which \(|z|=1\).

If \(m=0\), then \(|z|=1\) for every \(t>0\). Hence we may assume \(m\neq 0\). In this case,
\[
\begin{aligned}
z^2
&=
e^{\frac{2\pi im}{k}}
\left(e^{2i\theta}+\frac{1}{t}\right)
-\frac{1}{t} \\
&=
e^{\left(2\theta+\frac{2\pi m}{k}\right)i}
+
\frac{1}{t}e^{\frac{2\pi im}{k}}
-
\frac{1}{t}.
\end{aligned}
\]

Consequently,
\[
\begin{aligned}
|z|^4
&=
\left(
\cos\left(2\theta+\frac{2\pi m}{k}\right)
+
\frac{1}{t}\cos\left(\frac{2\pi m}{k}\right)
-
\frac{1}{t}
\right)^2
\\
&\quad +
\left(
\sin\left(2\theta+\frac{2\pi m}{k}\right)
+
\frac{1}{t}\sin\left(\frac{2\pi m}{k}\right)
\right)^2 \\
&=
1
+
\frac{4}{t^2}
\sin^2\left(\frac{\pi m}{k}\right)
+
\frac{4}{t}
\sin\left(
2\theta+\frac{\pi m}{k}
\right)
\sin\left(\frac{\pi m}{k}\right).
\end{aligned}
\]

Thus, $|z|=1$ if and only if
\[
\frac{4}{t^2}
\sin^2\left(\frac{\pi m}{k}\right)
+
\frac{4}{t}
\sin\left(
2\theta+\frac{\pi m}{k}
\right)
\sin\left(\frac{\pi m}{k}\right)
=
0.
\]

Equivalently,
\[
\frac{1}{t}
=
-
\frac{
\sin\left(
2\theta+\frac{\pi m}{k}
\right)
}{
\sin\left(\frac{\pi m}{k}\right)
}.
\]

We now determine the maximal possible range of \(t\). Since
$-1
\leq
-
\sin\left(
2\theta+\frac{\pi m}{k}
\right)
\leq
1,$ 
and $\sin\left(\frac{\pi m}{k}\right)>0,$ for $m=1,2,\dots,k-1,$ it follows that
\[
\frac{-1}{\sin\left(\frac{\pi m}{k}\right)}
\leq
\frac{1}{t}
\leq
\frac{1}{\sin\left(\frac{\pi m}{k}\right)}.
\]

Since \(t>0\), this implies $t\geq\sin\left(\frac{\pi m}{k}\right)$. Hence, for each $m\in \{1,2,\dots,k-1\},$ there exists \(a\in \partial\mathbb{D}\) such that the equation $g_t(z)=g_t(a)$ has a solution satisfying \(|z|=1\) whenever $t\geq\sin\left(\frac{m\pi}{k}\right).$ Consequently, for every \(a\in \partial\mathbb{D}\), any solution \(z\) of $g_t(z)=g_t(a)$ other than \(\pm a\) satisfies $|z|\neq 1$ provided
\[
t
<
\min\left\{
\sin\left(\frac{m\pi}{k}\right)
:
m=1,2,\dots,k-1
\right\}
=
\sin\left(\frac{\pi}{k}\right).
\]
This completes the proof of Lemma~\ref{L:Roots_g_t}.  
\end{proof}

\begin{theorem}\label{lll}
For $k\geq 3$, and for every \(\delta>0\), the polynomial hull of $S_t \cap \overline{\mathbb{B}(0,\delta)}$ contains a \((2k-3)\)-parameter family of analytic discs with boundary contained in \(S_t\), provided $0<t<\min\left\{\sin\left(\frac{\pi}{k}\right),\frac{1}{k-2}\right\}.$ Moreover, this family fills an open neighborhood of \(0\in \mathbb{C}^2\).
\end{theorem}
\begin{proof}
  Now we use part \((ii)\) of Result~\ref{llll}. Consider the curve on the unit circle defined by
\[
\mathcal{C}_t(z)
=
\frac{p_k(z,\bar z)}{z^k}
=
\sum_{m=1}^{k} t^{m-1}\frac{\binom{k}{m}}{k}
\left(\frac{\bar z}{z}\right)^m
=
\sum_{m=1}^{k} t^{m-1}\frac{\binom{k}{m}}{k}\bar z^{2m},
\]
where we used the fact that \(|z|=1\).
Since \(t\in \mathbb{R}\), we have $\overline{\mathcal{C}_t(z)}=\mathcal{C}_t(\bar z).$

 By Lemma~\ref{L:Roots_g_t}, if $0\leq t<\sin\left(\frac{\pi}{k}\right)$ and \(z_1,z_2\in \partial\mathbb{D}\) satisfy $\mathcal{C}_t(z_1)=\mathcal{C}_t(z_2),$ then necessarily $z_1=-z_2.$ 
 
Hence, \(z_1\) and \(z_2\) divide the unit circle into two arcs, each of length \(\pi>\frac{\pi}{2}\).

Since
\[
\min\left\{
\sin\left(\frac{\pi}{k}\right),
\frac{1}{k-2}
\right\}
=
\frac{1}{k-2},
\qquad k> 3,
\]
part \((ii)\) of Result~\ref{llll} yields the desired conclusion.  
\end{proof}

Wiegerinck~\cite{Wiegerinck} showed that the attachment of analytic discs near an isolated CR singularity of a totally real surface persists under sufficiently small perturbations. Consequently, the above results remain valid when \(S_t\) is replaced by \(M_t\). We obtain the proof of our main theorem.

\begin{proof}[Proof of Theorem \ref{ppp}]
    As computed in the proof of Theorem~\ref{T:Not_Poly_Convx},
\[
\operatorname{Ind}_{\mathcal{M}}(S_t,0)
=
\begin{cases}
k-1, & \text{if } t<1,\\[4pt]
-(k-1), & \text{if } t>1.
\end{cases}
\]
We also know that
\[
\phi_k(z)
=
\operatorname{Re}\left(\frac{p_k(z,\bar z)}{z^{k-2}}\right)
\]
is strictly subharmonic in a punctured neighborhood of the origin if
$0<t<\frac{1}{k-2}$. Since a small $\mathcal{C}^k$-perturbation of a strictly subharmonic function is again strictly subharmonic and $R(z,\bar z)=o(|z|^k)$ near the origin, we obtain that $\operatorname{Re}\!\left(\frac{p_k(z,\bar z)+R(z,\bar z)}{z^{k-2}}\right)$ is strictly subharmonic in a small deleted neighborhood of the origin. Therefore, by Result~\ref{th1}, we conclude that there is an analytic disc with boundary in $M_t$. Hence, $M_t$ is not locally polynomially convex at the origin if
$0<t<\frac{1}{k-2}$.

Furthermore, by the proof of Theorem~\ref{lll}, the curve $\mathcal{C}_t$ satisfies the hypotheses of Result~\ref{p} whenever
$t<\min\left\{\sin\left(\frac{\pi}{k}\right),\frac{1}{k-2}\right\}$.
Therefore, Result~\ref{p} implies that the polynomially convex hull of $M_t$ contains a $(2k-3)$-parameter family of analytic discs with boundaries contained in $M_t$ if
$0<t<\min\left\{\sin\left(\frac{\pi}{k}\right),\frac{1}{k-2}\right\}$.
Moreover, this family fills an open neighborhood of the origin in $\mathbb{C}^2$.
\end{proof}

\begin{proof}[Proof of Theorem~\ref{thm-k4}]  
 We use Result~$\ref{coro1}$ to prove this theorem.
	Let $\varphi_t(z,\bar z)
	=
	\operatorname{Re}\left(\frac{p(z,\bar z)}{z^2}\right), z=re^{i\theta}.$
	Then
	\[
	\begin{aligned}
		\frac{\partial^2 \varphi_t(z,\bar z)}
		{\partial z\,\partial \bar z}
		&=
		1-\sum^{4}_{m=2}t^{m-1}\frac{m(m-2)}{8}\binom{4}{m}
		\left[
		\left(\frac{z}{\bar z}\right)^{m-1}
		+
		\left(\frac{\bar z}{z}\right)^{m-1}
		\right] \\
		&=
		1-3t^2\cos(4\theta)-2t^3\cos(6\theta) \\
		&=: f(t,\theta).
	\end{aligned}
	\]
	
	Now,
	$
	f\left(t,\frac{\pi}{2}\right)
	=
	1-3t^2+2t^3
	>
	0,
	\forall\, t\in
	[\frac{1}{\sqrt{2}},1).
	$
	By continuity of \(f\), for every $t\in [\frac{1}{\sqrt{2}},1),
	$
	there exists \(\varepsilon_t>0\) such that
	$
	f(t,\theta)>0,
	\forall\,\theta\in
	\left(
	\frac{\pi}{2}-\varepsilon_t,
	\frac{\pi}{2}+\varepsilon_t
	\right)
	=:\tilde{E_t}.
	$
	
	Let $E_t
	=
	\{z\in\mathbb{C}:Arg(z)\in \tilde{E_t}\}.$
	Then \(\varphi_t(z,\bar z)\) is strictly subharmonic on the sector \(E_t\). 
	Our aim is to show that there exist $z_1,z_2\in E_t\cap \partial\mathbb{D}$ such that $\mathcal{C}_t(z_1)=\mathcal{C}_t(z_2),$ and $|Arg(z_1)-Arg(z_2)|<\frac{\pi}{2},$
	where $\mathcal{C}_t(z)
	=
	\frac{p(z,\bar z)}{z^4},  z\in \partial\mathbb{D}.$
	Consider
	\[
	g_t(z)
	=
	\sum_{m=1}^{4}
	t^{m-1}\frac{\binom{4}{m}}{4}z^{2m}.
	\]
	Then $g_t(z)=
	\frac{1}{4t}
	\left[(tz^2+1)^4-1\right].$
	We know that for every
	$
	1> t\geq \sin\left(\frac{\pi}{4}\right),
	$
	there exists \(a=e^{i\theta_0}\in \partial\mathbb{D}\) such that
	$
	g_t(z)=g_t(a)
	$
	has two additional solutions on the unit circle other than \(a\) and \(-a\). These solutions are obtained from
	\begin{equation}\label{equ}
		\sin\left(2\theta+\frac{\pi}{4}\right)
		=
		-\frac{\sin(\pi/4)}{t}.
		\tag{$\ast\ast\ast$}
	\end{equation}
	
Since $\frac{1}{\sqrt{2}}\leq t<1,$ that implies $\theta\in (\frac{\pi}{2},\frac{5\pi}{8}]\cup [\frac{5\pi}{8},\frac{3\pi}{4}).$
Since we are interested in solutions near \(\pi/2\), we choose
$
	\theta_0\in
	\left(
	\frac{\pi}{2},
	\frac{5\pi}{8}
	\right].
$
Equation (\ref{equ}) also implies
	 $\sin(2\theta)
	=
	\frac{-1\pm \sqrt{2t^2-1}}{2t},$
	and consequently, $\cos(2\theta)
	=
	\frac{-1\mp \sqrt{2t^2-1}}{2t}.$
There are two possible choices of $(\sin(2\theta),\cos(2\theta))$.

	Since $	\theta_0\in
	\left(
	\frac{\pi}{2},
	\frac{5\pi}{8}
	\right],$
	we obtain
	\[
	\sin(2\theta_0)
	=
	\frac{-1+\sqrt{2t^2-1}}{2t},
	\qquad
	\cos(2\theta_0)
	=
	\frac{-1-\sqrt{2t^2-1}}{2t}.
	\]
The remaining two solutions of $g_t(z)=g_t(a)$	on the unit circle other than $\pm a$,  are determined by
	\[
	\begin{aligned}
		z^2
		=
		e^{\frac{2\pi i}{4}}
		\left(e^{2i\theta_0}+\frac{1}{t}\right)
		-\frac{1}{t} 
		=e^{(2\pi-2\theta_0)i}.
	\end{aligned}
	\]
\begin{claim}
	$	z_1=e^{i\theta_0}$ and $z_2=e^{(\pi-\theta_0)i}.$
\end{claim}
Now, $|\theta_0-(\pi-\theta_0)|
=
|2\theta_0-\pi|
<
\frac{\pi}{2},$ since $\theta_0\in
\left(
\frac{\pi}{2},
\frac{5\pi}{8}
\right].$
It therefore remains to show that $z_1,z_2\in E_t.$ Since $	f(t,\theta)=f(t,\pi-\theta), \forall\,\theta\in\mathbb{R},$ it suffices to prove that
$\theta_0\in \tilde{E_t}.$

Now, $f\!\left(t,\frac{\pi}{2}\right)>0,
\forall\, t\in
[\frac1{\sqrt2},1),$
and
	\[
	\begin{aligned}
		f(t,\theta_0)
		&=
		1-3t^2\cos(4\theta_0)-2t^3\cos(6\theta_0) \\
		&= -1+3t^2-(1+t^2)\sqrt{2t^2-1}.
	\end{aligned}
	\]
Now, for each $	t\in
[\frac1{\sqrt2},1),$
if we can show that $	f(t,\theta)>0,
\forall\, \theta\in
\left(\frac{\pi}{2},\theta_0\right),$
then the proof is complete.
	
	Observe that
	\[
	f(t,\theta)
	=
	1+3t^2
	+6t^3\cos 2\theta
	-6t^2\cos^2 2\theta
	-8t^3\cos^3 2\theta.
	\]
	
	Since $	2\theta_0\in
	\left(\pi,\frac{5\pi}{4}\right],$
	where \(\cos\) is strictly increasing. Let $	x=\cos 2\theta,$
	and define
	\[
	\tilde{f}(t,x)
	=
	1+3t^2+6t^3x-6t^2x^2-8t^3x^3.
	\]
	
	By an elementary calculation we can show that
	\[
	\tilde{f}(t,x)>0,
	\qquad
	\forall\, x\in
	\left(-1,\frac{-1-\sqrt{2t^2-1}}{2t}\right).
	\]
Hence, this  proves the theorem.
	
\end{proof}

\begin{remark}
Suppose $t<\sin\left(\frac{\pi}{4}\right).$ Then the equation \[\sin\left(2\theta+\frac{\pi}{4}\right)
=
-\frac{\sin(\pi/4)}{t}\] has no solution.
Hence, if \(z_1,z_2\in S^1\) satisfy $\mathcal{C}_t(z_1)=\mathcal{C}_t(z_2),$ the only possible relation is $z_1=-z_2.$
However,
\[
|Arg(z_1)-Arg(z_2)|\nleq \frac{\pi}{2}.
\]
Therefore Result~\ref{coro1} cannot be applied when $t\in\left[\frac12,\sin\left(\frac{\pi}{4}\right)\right).$
\end{remark}

\begin{remark}
Assume that \(k\geq 5\). Then there exist sectors on which $\operatorname{Re}\left(\frac{p_k(z,\bar z)}{z^{\,j-1}}\right)$ is strictly subharmonic.
If $t<\sin\left(\frac{\pi}{k}\right),$ then, similarly as above, the only possibility is $z_2=-z_1,$ which again does not allow us to apply the corollary.
On the other hand, if $\sin\left(\frac{\pi}{k}\right)\leq t,$ then, by arguments analogous to those used in the previous theorem, we obtain additional solutions of the equation $\mathcal{C}_t(z_1)=\mathcal{C}_t(z_2).$
Nevertheless, these solutions do not lie in the same sector in which $\operatorname{Re}\left(\frac{p_k(z,\bar z)}{z^{\,j-1}}\right)$ is strictly subharmonic.
\end{remark}

\section{Local polynomial convexity}
In this section, we investigate the local polynomial convexity of $M_t$ at the origin. We begin by studying the surface $S_t$, which leads to the following theorem.
\begin{theorem}\label{T:Without error term}
Fix \(k \in \mathbb{N}\) with \(k\geq3\). Then the surface \(S_t\) is locally polynomially convex at the origin if $t > cosec\left(\frac{\pi}{k}\right).$ 
\end{theorem}
\begin{proof}
Let us consider the proper holomorphic map $\Psi: \mathbb{C}^2 \to \mathbb{C}^2$ defined by  $ \Psi(z,w) = (z, w^k).$ Then, we have  
  \[\Psi^{-1}\left(S_t \cap \overline{B(0,\delta)}\right) = \bigcup_{j=0}^{k-1} S_j^t,\]
  where  $ S_j^t := \left\{(z,w) \in \mathbb{C}^2 : w = \alpha^j (z + t\bar{z})\right\}\cap \overline{P(0,\delta_1)},$ $\alpha=e^{\frac{2\pi i}{k}}$, and $P(0,\delta_1)$ is a polydisk of radius $\delta_1=\left(\frac{\delta}{\sqrt{2}}, (\frac{\delta}{\sqrt{2}})^{\frac{1}{k}}\right).$ 

  To apply Result~\ref{R:Sto_apprx}, we consider the holomorphic polynomial $p(z,w) = zw$. Now, for $(z,w) \in S_0^t$,
	\[
	p(z,w) = z(z + t\bar{z}) = z^2 + t|z|^2.
	\]
	Let $z = r e^{i\theta}$. Then
	\[
	p(S_0^t) = \left\{ r^2 \bigl(t + e^{2i\theta}\bigr) : r \geq 0,\ \theta \in [0,2\pi] \right\}\cap \overline{D(0,\delta_2)}, \qquad \delta_2=\left(\frac{\delta}{\sqrt{2}}\right)^{\frac{1}{k}+1}.
	\]
	\noindent Clearly, for fixed $t$, $p(S_0^t)$ is a union of circles centered on the real axis.  
	As $t$ varies, $p(S_0^t)$ lies in a sector in the complex plane containing the positive real axis, say $\omega_0$. Similarly, each $p(S_j^t)$ is a sector $\omega_j$  obtained by rotating $\omega_0$ by an angle $\frac{2\pi j}{k}$.

    \smallskip
    
	\noindent Our goal is to determine the values of $t$ for which these $k$ sectors intersect only at the origin.  
	To compute the angular width of each sector, it suffices to compute the angle of the tangent to the circle $\left\{r^2\bigl(t + e^{2i\theta}\bigr):\ \theta \in [0,2\pi]\right\}$ that passes through the origin.
	Let $\varphi$ denote this angle. Then $2\varphi$ is the angular width of $p(S_0^t)$.

    \smallskip
	
	Let us define $f(\theta) = r^2\bigl(t + e^{2i\theta}\bigr),$ then $f'(\theta) = 2i r^2 e^{2i\theta}.$ The tangent line at $\theta$ is given by
	\[
	\ell(s) := r^2\bigl(t + e^{2i\theta}\bigr) + s \cdot 2i r^2 e^{2i\theta}, \quad s \in \mathbb{R}.
	\]
	\noindent Suppose that $\ell(s)$ passes through the origin. Then there exists $s_0 \in \mathbb{R}$ such that $t + e^{2i\theta} + 2i s_0 e^{2i\theta} = 0.$ Separating real and imaginary parts, we obtain
	\[
	t + \cos(2\theta) - 2 s_0 \sin(2\theta) = 0,
	\]
	\[
	\sin(2\theta) + 2 s_0 \cos(2\theta) = 0.
	\]
	From these, we deduce
	\[
	\cos(2\theta) = -\frac{1}{t}.
	\]
	
	Now define $\varphi$ by
	\[
	\tan \varphi = \frac{\sin(2\theta)}{t + \cos(2\theta)}.
	\]
	Using $\cos(2\theta) = -\frac{1}{t}$, we compute
	\[
	\tan \varphi 
	= \frac{\sin(2\theta)}{\cos^2(2\theta) - 1}
	= \frac{\sin(2\theta)\cos(2\theta)}{\cos^2(2\theta) - 1}
	= -\cot(2\theta).
	\]
	
\noindent Each sector $p(S_j^t)$ has uniform angular length $2\varphi$. To ensure that these sectors fit into $k$ disjoint sectors, we require 
    $2\varphi < \frac{2\pi}{k},
\ i.e. \ 
	\varphi < \frac{\pi}{k}.$ Thus, $\tan \varphi < \tan\left(\frac{\pi}{k}\right).$ Using $\tan\varphi = -\cot(2\theta)$, we get $-\cot(2\theta) < \tan\left(\frac{\pi}{k}\right).$ This leads to
	\[
	\frac{1}{\sqrt{t^2 - 1}} < \tan\left(\frac{\pi}{k}\right),
	\]
	and hence
	\[
	t >\operatorname{cosec}\left(\frac{\pi}{k}\right).
	\]

  \noindent Furthermore,  
  \begin{align*}
    p^{-1}\{0\} \cap \left(\bigcup_{j=0}^{k-1} S_j^t\right) = \{0\}
  \end{align*}  
  is polynomially convex.  

  \smallskip  

  \noindent Thus, we have established the following:  
  \begin{itemize}
    \item $p(S_j^t) \subset \omega_j$, where $\omega_j$ are angular sectors, and $\omega_j \cap \omega_l = \{0\},\quad j\ne l$.
    \item $p^{-1}\{0\} \cap \left(\bigcup_{j=0}^{k-1} S_j^t\right) = \{0\}$ is polynomially convex.
  \end{itemize}  

  \noindent Therefore, Result~\ref{R:Sto_apprx} gives that $\Psi^{-1}(S_t)$ is locally polynomially convex at the origin if $t >\operatorname{cosec}\left(\frac{\pi}{k}\right).$  Since $\Psi$ is a proper holomorphic map, we apply Result~\ref{R:Sto_propr} to conclude that  the surface \(S_t\) is locally polynomially convex at the origin if $t > cosec\left(\frac{\pi}{k}\right).$ 
\end{proof}

\begin{proof}[Proof of Theorem \ref{T:With error term}]
	Consider the proper holomorphic map $\Psi: \mathbb{C}^2 \to \mathbb{C}^2$ defined by  $\Psi(z,w) = (z, w^k).$ We have,	
\[
M_t=\left\{(z,w):\, w=(z+t\bar z)^k+o(|z|^k)\right\}\cap \overline{B(0,\delta)}.
\]

\noindent Then
\[
\Psi^{-1}(M_t)=\bigcup_{j=0}^{k-1} M_j^t,
\]
where
$M_j^t
=
\left\{
(z,w)\in \mathbb C^2:\,
w=\alpha^j\bigl((z+t\bar z)+H(z)\bigr)
\right\}
\cap \overline{P(0,\delta_1)},$  and $P(0,\delta_1)$ is a polydisk of radius $\delta_1=\left(\frac{\delta}{\sqrt{2}}, (\frac{\delta}{\sqrt{2}})^{\frac{1}{k}}\right)$
with $H(z)=o(|z|),$ $\alpha=e^{\frac{2\pi i}{k}}.$

Consider the holomorphic polynomial $p(z,w) = zw$. We know that $S_0^t=\{(z,w): w=(z+t\bar z)\}\cap\overline{P(0,\delta_1)}.$ Let $(z,w)\in M_0^t$. Then
\[
p(z,w)
=
z(z+t\bar z)+zH(z)
=:Q(z)+R(z),
\]
where $Q(z)=z(z+t\bar z)$ and $
R(z)=zH(z).$ For $\delta>0$, define
\[
m_\delta
=
\sup_{0<|z|<\frac{\delta}{\sqrt{2}}}\frac{|H(z)|}{|z|}.
\]
Since $H(z)=o(|z|)$, we have $m_\delta\to 0$ as $\delta\to0$.
Also, for $|z|<\frac{\delta}{\sqrt{2}}$,
\[
|R(z)|
\le m_\delta |z|^2,
\quad
|Q(z)|
=
|z|^2\,|e^{2i\theta}+t|
\ge (t-1)|z|^2.
\]

\noindent Hence
\[
\left|\frac{R(z)}{Q(z)}\right|
\le
\frac{m_\delta}{t-1},
\qquad 0<|z|<\frac{\delta}{\sqrt{2}}.
\]
\noindent Now, $p(z,w)=Q(z)\left(1+\frac{R(z)}{Q(z)}\right),$
and therefore
\[
|\arg p(z,w)-\arg Q(z)|
=
\left|\arg\left(1+\frac{R(z)}{Q(z)}\right)\right|.
\]

\noindent Using the elementary estimate $|\arg(1+\zeta)|\le 2|\zeta| $ for $|\zeta|<\frac12,$ and choosing $\delta>0$ sufficiently small so that $\frac{m_\delta}{t-1}<\frac12,$
we obtain
\[
|\arg p(z,w)-\arg Q(z)|
\le
2\left|\frac{R(z)}{Q(z)}\right|
\le
\frac{2m_\delta}{t-1}.
\]

\noindent Setting $\varepsilon_\delta
=
\frac{2m_\delta}{t-1}$ for a fixed $t>1$,
we conclude that
\[
|\arg p(z,w)-\arg Q(z)|
<
\varepsilon_\delta
\qquad \text{for all } (z,w)\in M_0^t.
\]

\noindent Since $m_\delta\to0$ as $\delta\to0$, it follows that $\varepsilon_\delta\to0\text{ as } \delta\to0.$ Note that $Q(z)\in p(S_0^t)$.

\smallskip

We have already seen that \(p(S_0^t)\) is inside a sector $\omega_0$, centered at the origin and the real axis divides the sector into two equal parts. Now \(p(M_0^t)\) is a perturbation of \(p(S_0^t)\). If \(2\varphi\) denotes the angular width of \(\omega_0\), then \(p(M_0^t)\) is contained in a sector $\tilde{\omega}_0$ $(\omega_0\subset \tilde{\omega}_0)$ with angular width is slightly larger; say it is $2(\varphi+\varepsilon_\delta),$ where the perturbation term \(\varepsilon_\delta\) depends on \(\delta\). Since each \(p(M_j^t)\) has the same angular width \(2(\varphi+\varepsilon_\delta)\), these sectors fit into \(k\) pairwise disjoint sectors $\tilde{\omega}_j$; provided that
\[
2(\varphi+\varepsilon_\delta)<\frac{2\pi}{k},
\]
that is,
\[
\varphi+\varepsilon_\delta<\frac{\pi}{k},
\]
or equivalently,
\[
\varphi<\frac{\pi}{k}-\varepsilon_\delta.
\]

\noindent By calculations similar to those carried out in the proof of Theorem~\ref{T:Without error term}, we obtain
\[
t>cosec\left(\frac{\pi}{k}-\varepsilon_\delta\right).
\]

	\noindent Furthermore,  
	\begin{align*}
		p^{-1}\{0\} \cap \left(\bigcup_{j=0}^{k-1} M_j^t\right) = \{0\}
	\end{align*}  
	is polynomially convex.  
	
	\smallskip  
	
	\noindent Thus, we have established the following:  
	\begin{itemize}
		\item $p(M_j^t) \subset \tilde{\omega}_j$, where $\tilde{\omega}_j$ are angular sectors, and $\tilde{\omega}_j \cap \tilde{\omega}_l = \{0\},\quad j\ne l$.
		\item $p^{-1}\{0\} \cap \left(\bigcup_{j=0}^{k-1} M_j^t\right) = \{0\}$ is polynomially convex.
	\end{itemize}  
	
	\noindent 
Therefore, by Result~\ref{R:Sto_apprx}, we conclude that $\Psi^{-1}(M_t)$ is locally polynomially convex at the origin if $t>cosec\left(\frac{\pi}{k}-\varepsilon_\delta\right).$  Since $\Psi$ is a proper holomorphic map, we apply Result~\ref{R:Sto_propr} to conclude that  the surface \(M_t\) is locally polynomially convex at the origin if $t>cosec\left(\frac{\pi}{k}-\varepsilon_\delta\right).$ 
   
	Let $t>cosec\left(\frac{\pi}{k}\right).$ Then there exists \(\delta>0\) such that the perturbation term \(\varepsilon_\delta>0\) satisfies
\[
t>cosec\left(\frac{\pi}{k}-\varepsilon_\delta\right)
>
cosec\left(\frac{\pi}{k}\right).
\]

Hence \(M_t\) is locally polynomially convex at the origin whenever
$t>cosec\left(\frac{\pi}{k}\right)$.
\end{proof}

\noindent {\bf Acknowledgements.}  Sushil Gorai is partially supported by an ARG MATRICS grant (ANRF/ARGM/2025/002958/MTR). Suman Karak acknowledges the financial support of the NBHM, Department of Atomic Energy (DAE), Government of India, through the NBHM Doctoral Fellowship (Reference No.~0203/11/2023/R\&D-II/DAE/13950).  Golam Mostafa Mondal acknowledges the financial support of the NBHM, DAE, Government of India, under the NBHM Postdoctoral Fellowship (Reference No. 0204/16(4)/2024/R$\&$D-II/6748, dated 09.05.2024).

	\bibliographystyle{plain}
	\bibliography{biblio.bib}

@incollection {Harris1991,
    AUTHOR = {Harris, G. A.},
     TITLE = {Degenerate surfaces in {${\bf C}^2$}},
 BOOKTITLE = {Several complex variables and complex geometry, {P}art 3
              ({S}anta {C}ruz, {CA}, 1989)},
    SERIES = {Proc. Sympos. Pure Math.},
    VOLUME = {52, Part 3},
     PAGES = {179--189},
 PUBLISHER = {Amer. Math. Soc., Providence, RI},
      YEAR = {1991},
      ISBN = {0-8218-1491-5},
   MRCLASS = {32F25},
  MRNUMBER = {1128592},
MRREVIEWER = {Serge\ Ivashkovich},
       DOI = {10.1090/pspum/052.3/1128592},
       URL = {https://doi.org/10.1090/pspum/052.3/1128592},
}

@book {Sto07,
	AUTHOR = {Stout, E. L.},
	TITLE = {Polynomial convexity},
	SERIES = {Progress in Mathematics},
	VOLUME = {261},
	PUBLISHER = {Birkh\"{a}user Boston, Inc., Boston, MA},
	YEAR = {2007},
	PAGES = {xii+439},
	ISBN = {978-0-8176-4537-3; 0-8176-4537-3},
	MRCLASS = {32E20},
	MRNUMBER = {2305474},
	MRREVIEWER = {Marshall A. Whittlesey},
}

@article {SG2011,
    AUTHOR = {Gorai, S.},
     TITLE = {Local polynomial convexity of the union of two totally-real
              surfaces at their intersection},
   JOURNAL = {Manuscripta Math.},
  FJOURNAL = {Manuscripta Mathematica},
    VOLUME = {135},
      YEAR = {2011},
    NUMBER = {1-2},
     PAGES = {43--62},
      ISSN = {0025-2611,1432-1785},
   MRCLASS = {32E20 (32V40)},
  MRNUMBER = {2783386},
MRREVIEWER = {Marshall\ A.\ Whittlesey},
       DOI = {10.1007/s00229-010-0405-x},
       URL = {https://doi.org/10.1007/s00229-010-0405-x},
}

@inproceedings {Kallin65,
	AUTHOR = {Kallin, E.},
	TITLE = {Fat polynomially convex sets},
	BOOKTITLE = {Function {A}lgebras ({P}roc. {I}nternat. {S}ympos. on
	{F}unction {A}lgebras, {T}ulane {U}niv., 1965)},
	PAGES = {149--152},
	PUBLISHER = {Scott-Foresman, Chicago, Ill.},
	YEAR = {1966},
	MRCLASS = {32.70 (32.49)},
	MRNUMBER = {0194618},
	MRREVIEWER = {A. B. Willcox},
}

@article {Bharali2012,
    AUTHOR = {Bharali, G.},
     TITLE = {The local polynomial hull near a degenerate {CR} singularity:
              {B}ishop discs revisited},
   JOURNAL = {Math. Z.},
  FJOURNAL = {Mathematische Zeitschrift},
    VOLUME = {271},
      YEAR = {2012},
    NUMBER = {3-4},
     PAGES = {1043--1063},
      ISSN = {0025-5874},
   MRCLASS = {32E20 (32V20 46J10)},
  MRNUMBER = {2945596},
MRREVIEWER = {I. G. Kossovskiy},
       DOI = {10.1007/s00209-011-0902-y},
       URL = {https://doi.org/10.1007/s00209-011-0902-y},
}

@article {Bharali2011,
    AUTHOR = {Bharali, G.},
     TITLE = {Polynomial approximation, local polynomial convexity, and
              degenerate {CR} singularities---{II}},
   JOURNAL = {Internat. J. Math.},
  FJOURNAL = {International Journal of Mathematics},
    VOLUME = {22},
      YEAR = {2011},
    NUMBER = {12},
     PAGES = {1721--1733},
      ISSN = {0129-167X,1793-6519},
   MRCLASS = {32E20 (30E10 32A65 32V20)},
  MRNUMBER = {2872529},
MRREVIEWER = {I.\ G.\ Kossovskiy},
       DOI = {10.1142/S0129167X11007446},
       URL = {https://doi.org/10.1142/S0129167X11007446},
}

@article {Bharali2005,
    AUTHOR = {Bharali, G.},
     TITLE = {Surfaces with degenerate {CR} singularities that are locally
              polynomially convex},
   JOURNAL = {Michigan Math. J.},
  FJOURNAL = {Michigan Mathematical Journal},
    VOLUME = {53},
      YEAR = {2005},
    NUMBER = {2},
     PAGES = {429--445},
      ISSN = {0026-2285,1945-2365},
   MRCLASS = {32E20 (46J10)},
  MRNUMBER = {2152709},
MRREVIEWER = {Marshall\ A.\ Whittlesey},
       DOI = {10.1307/mmj/1123090777},
       URL = {https://doi.org/10.1307/mmj/1123090777},
}

@article {Bharali2006,
    AUTHOR = {Bharali, G.},
     TITLE = {Polynomial approximation, local polynomial convexity, and
              degenerate {CR} singularities},
   JOURNAL = {J. Funct. Anal.},
  FJOURNAL = {Journal of Functional Analysis},
    VOLUME = {236},
      YEAR = {2006},
    NUMBER = {1},
     PAGES = {351--368},
      ISSN = {0022-1236,1096-0783},
   MRCLASS = {32E20 (32V20 46J10)},
  MRNUMBER = {2227137},
MRREVIEWER = {Marshall\ A.\ Whittlesey},
       DOI = {10.1016/j.jfa.2006.02.001},
       URL = {https://doi.org/10.1016/j.jfa.2006.02.001},
}

@article {Wiegerinck,
    AUTHOR = {Wiegerinck, J.},
     TITLE = {Local polynomially convex hulls at degenerated {CR}
              singularities of surfaces in {$\mathbb{C}^2$}},
   JOURNAL = {Indiana Univ. Math. J.},
  FJOURNAL = {Indiana University Mathematics Journal},
    VOLUME = {44},
      YEAR = {1995},
    NUMBER = {3},
     PAGES = {897--915},
      ISSN = {0022-2518},
   MRCLASS = {32E20 (32F25)},
  MRNUMBER = {1375355},
MRREVIEWER = {Gary A. Harris},
       DOI = {10.1512/iumj.1995.44.2014},
       URL = {https://doi.org/10.1512/iumj.1995.44.2014},
}

@article {MR200476,
    AUTHOR = {Bishop, E.},
     TITLE = {Differentiable manifolds in complex {E}uclidean space},
   JOURNAL = {Duke Math. J.},
  FJOURNAL = {Duke Mathematical Journal},
    VOLUME = {32},
      YEAR = {1965},
     PAGES = {1--21},
      ISSN = {0012-7094,1547-7398},
   MRCLASS = {32.40},
  MRNUMBER = {200476},
MRREVIEWER = {P.\ Dolbeault},
       URL = {http://projecteuclid.org/euclid.dmj/1077375631},
}

@incollection {MR1128527,
    AUTHOR = {Thomas, P. J.},
     TITLE = {Unions minimales de {$n$}-plans r\'eels d'enveloppe \'egale
              \`a{} {$\mathbb{C}^n$}},
 BOOKTITLE = {Several complex variables and complex geometry, {P}art 1
              ({S}anta {C}ruz, {CA}, 1989)},
    SERIES = {Proc. Sympos. Pure Math.},
    VOLUME = {52, Part 1},
     PAGES = {233--244},
 PUBLISHER = {Amer. Math. Soc., Providence, RI},
      YEAR = {1991},
      ISBN = {0-8218-1489-3},
   MRCLASS = {32E20},
  MRNUMBER = {1128527},
MRREVIEWER = {Ahmed\ Zeriahi},
       DOI = {10.1090/pspum/052.1/1128527},
       URL = {https://doi.org/10.1090/pspum/052.1/1128527},
}

@article {MR1115074,
    AUTHOR = {Forstneri\v{c}, F. and Stout, E. L.},
     TITLE = {A new class of polynomially convex sets},
   JOURNAL = {Ark. Mat.},
  FJOURNAL = {Arkiv f\"or Matematik},
    VOLUME = {29},
      YEAR = {1991},
    NUMBER = {1},
     PAGES = {51--62},
      ISSN = {0004-2080,1871-2487},
   MRCLASS = {32E20 (32D15)},
  MRNUMBER = {1115074},
MRREVIEWER = {Guido\ Lupacciolu},
       DOI = {10.1007/BF02384330},
       URL = {https://doi.org/10.1007/BF02384330},
}

@article {MR1488338,
    AUTHOR = {J\"oricke, B.},
     TITLE = {Local polynomial hulls of discs near isolated parabolic
              points},
   JOURNAL = {Indiana Univ. Math. J.},
  FJOURNAL = {Indiana University Mathematics Journal},
    VOLUME = {46},
      YEAR = {1997},
    NUMBER = {3},
     PAGES = {789--826},
      ISSN = {0022-2518,1943-5258},
   MRCLASS = {32E20},
  MRNUMBER = {1488338},
MRREVIEWER = {Norman\ Levenberg},
       DOI = {10.1512/iumj.1997.46.1501},
       URL = {https://doi.org/10.1512/iumj.1997.46.1501},
}

@article {MR773068,
    AUTHOR = {Harris, G. A.},
     TITLE = {Lowest order invariants for real-analytic surfaces in {${\mathbb{C}}^2$}},
   JOURNAL = {Trans. Amer. Math. Soc.},
  FJOURNAL = {Transactions of the American Mathematical Society},
    VOLUME = {288},
      YEAR = {1985},
    NUMBER = {1},
     PAGES = {413--422},
      ISSN = {0002-9947,1088-6850},
   MRCLASS = {32F25},
  MRNUMBER = {773068},
MRREVIEWER = {Tosiaki\ Kori},
       DOI = {10.2307/2000447},
       URL = {https://doi.org/10.2307/2000447},
}

@article{MR0960837,
   author={Weinstock, B. M.},
   title={On the polynomial convexity of the union of two maximal totally
   real subspaces of $\mathbb{C}^n$},
   journal={Math. Ann.},
   volume={282},
   year={1988},
   number={1},
   pages={131--138},
   issn={0025-5831},
   review={\MR{0960837}},
   doi={10.1007/BF01457016},
}

@article{MR1832326,
   author={Garc\'ia, A. S. },
   title={Polynomial hulls of smooth discs: a survey},
   journal={Irish Math. Soc. Bull.},
   year={2000},
   date={2000},
   pages={135--153},
   review={\MR{1832326}},
}

@article{MR0894560,
   author={Forstneri\v{c}, F.},
   title={Analytic disks with boundaries in a maximal real submanifold of
   $\mathbb{C}^2$},
   language={English, with French summary},
   journal={Ann. Inst. Fourier (Grenoble)},
   volume={37},
   year={1987},
   number={1},
   pages={1--44},
   issn={0373-0956},
   review={\MR{0894560}},
   doi={10.5802/aif.1076},
}

@article{gorai2025certainrealsurfacesmathbbc2,
      title={Certain real surfaces in $\mathbb{C}^2$ with isolated singularities}, 
      JOURNAL = {to appear in Ann. Inst. Fourier (Grenoble)},
      author={Gorai, S.},
     year={2026},
      eprint={1909.04085},
      archivePrefix={arXiv},
      primaryClass={math.CV},
      url={https://arxiv.org/abs/1909.04085}, 
}

@article {MR3206683,
    AUTHOR = {Gorai, S.},
     TITLE = {A note on polynomial convexity of the union of finitely many
              totally-real planes in {$\mathbb{C}^2$}},
   JOURNAL = {J. Math. Anal. Appl.},
  FJOURNAL = {Journal of Mathematical Analysis and Applications},
    VOLUME = {418},
      YEAR = {2014},
    NUMBER = {2},
     PAGES = {842--851},
      ISSN = {0022-247X,1096-0813},
   MRCLASS = {32E20 (32V40)},
  MRNUMBER = {3206683},
MRREVIEWER = {Serge\ Ivashkovich},
       DOI = {10.1016/j.jmaa.2014.03.011},
       URL = {https://doi.org/10.1016/j.jmaa.2014.03.011},
}

@article {MR3123672,
    AUTHOR = {Gorai, S.},
     TITLE = {On the polynomial convexity of the union of three totally real
              planes in {$\mathbb{C}^2$}},
   JOURNAL = {Int. Math. Res. Not. IMRN},
  FJOURNAL = {International Mathematics Research Notices. IMRN},
      YEAR = {2013},
    NUMBER = {21},
     PAGES = {4985--5001},
      ISSN = {1073-7928,1687-0247},
   MRCLASS = {32E20},
  MRNUMBER = {3123672},
MRREVIEWER = {I.\ G.\ Kossovskiy},
       DOI = {10.1093/imrn/rns200},
       URL = {https://doi.org/10.1093/imrn/rns200},
}

@inproceedings {MR769507,
    AUTHOR = {O'Farrell, A. G. and Preskenis, K. J. and Walsh, D.},
     TITLE = {Holomorphic approximation in {L}ipschitz norms},
 BOOKTITLE = {Proceedings of the conference on {B}anach algebras and several
              complex variables ({N}ew {H}aven, {C}onn., 1983)},
    SERIES = {Contemp. Math.},
    VOLUME = {32},
     PAGES = {187--194},
 PUBLISHER = {Amer. Math. Soc., Providence, RI},
      YEAR = {1984},
      ISBN = {0-8218-5034-2},
   MRCLASS = {32E30 (32E05)},
  MRNUMBER = {769507},
MRREVIEWER = {Wies\l aw\ Ple\'sniak},
       DOI = {10.1090/conm/032/769507},
       URL = {https://doi.org/10.1090/conm/032/769507},
}

\end{document}